\documentclass[a4paper,10pt]{article}
\usepackage{amsfonts}
\usepackage{amssymb}
\usepackage{latexsym}
\usepackage{amsmath}
\usepackage{amscd}
\usepackage{amsthm}
\usepackage{hyperref}

\usepackage{url}
\usepackage{xcolor}

\newtheorem{theorem}{Theorem}[section]
\newtheorem{definition}[theorem]{Definition}

\newtheorem{proposition}[theorem]{Proposition}
\newtheorem{lemma}[theorem]{Lemma}
\newtheorem{remark}[theorem]{Remark}
\newtheorem{corollary}[theorem]{Corollary}

\newcommand{\NN}{{\mathcal{N}}}

\newcommand{\E}{{\mathbb E}}
\newcommand{\N}{{\mathbb N}}
\newcommand{\Z}{{\mathbb Z}}
\renewcommand{\P}{{\mathbb P}}
\newcommand{\Q}{{\mathbb Q}}
\newcommand{\R}{{\mathbb R}}

\newcommand{\x}{{\mathbf x}}
\newcommand{\y}{{\mathbf y}}

\newcommand{\z}{{\mathbf z}}

\newcommand{\WW}{{\mathcal W}}

\title{Convergence of a drainage network model with branching to the Brownian net}

\usepackage{authblk}

\author[]{Krishnamurthi Ravishankar\thanks{ravi101048@gmail.com $^{\dagger}$kumarjit.saha@ashoka.edu.in} }
\author[]{Kumarjit Saha$^{\dagger}$}
\affil[]{$^{\ast}$SUNY-New Paltz, 1 Hawk drive, New Paltz, NY 12561, USA}
\affil[]{$^{\dagger}$Department of Mathematics, Ashoka University Sonipat, Haryana, India}
\date{\today }

\begin{document}

\maketitle

\begin{abstract}
We consider a drainage network model represented by a system of coalescing random walks exhibiting long range dependence before coalescence. This model was introduced by Gangopadhyay et. al. \cite{GRS04} and Coletti et. al. \cite{CFD09} proved its convergence to the Brownian web under diffusive scaling. Santos et.al. \cite{SVZ24} considered a branching version of this model where branching occurs with low probabilities varying with the scaling parameter and showed that any sub-sequential limit contains a copy of the Brownian net. It was conjectured that the limit is indeed the Brownian net. In this work we prove this conjecture. Further, compared to \cite{SVZ24}, we explore the question of scaling limit
in a wider spectrum of scaling regime. This work contributes to the broader 
understanding of the basin of attraction of the Brownian net.
\end{abstract}

\section{Introduction and our result}
\label{sec:Intro}

The Brownian net, introduced in \cite{SS08}, arises as the diffusive 
scaling limit of a system of branching and coalescing one-dimensional 
simple symmetric random walks where walks are independent till 
the time of coalescence. The Brownian net is a natural extension of the Brownian web, 
which originated in the work of \cite{FINR04} while dealing with the scaling limit 
of coalescing simple symmetric random walk paths. Loosely speaking, the Brownian 
web can be described as coalescing Brownian motions starting from everywhere in the space-time plane of $\R^2$. The system of coalescing Brownian motions starting from everywhere had been considered earlier too, 
initially by Arratia \cite{A79}, \cite{A81} and later by T\'{o}th and Werner \cite{TW98}. In 
the seminal work of \cite{FINR04} a nice Polish space had been suggested and convergence conditions to the Brownian web were developed. Since then, the basin of attraction of the Brownian web has become an object of extensive study (see \cite{FFW05}, \cite{NRS05},
\cite{CFD09}, \cite{CV14}, \cite{RSS16A}, \cite{CSST21}, \cite{VZ23}). In this work we 
study convergence in distribution for a diffusively scaled branching drainage network 
model to the Brownian net. The speciality of our model is that the system of 
one-dimensional random walks starting from a random subset of the space and
time lattice $\Z \times \Z$ exhibits long range dependence before coalescence.

Drainage network models are used to explain various scaling relations 
empirically observed for many river networks.  
Scheidegger \cite{S67} was the first to propose a directed network where
each source empties in one of its two neighbours in a preferred direction. 
Howard \cite{H71} removed the restriction of drainage to neighbouring sites only
and proposed a more general network model which was later analysed by Gangopadhyay et. al. \cite{GRS04} (see also \cite{CFD09} and \cite{RSS16B}). This random directed 
network has been referred here as the Howard's network which gives a 
natural model of one dimensional coalescing random walks with long range 
interactions and non-crossing paths. Coletti et. al. \cite{CFD09} proved 
that under diffusive scaling, this model belongs to the basin of attraction of 
the Brownian web. This result was later strengthened in two subsequent works, 
 \cite{RSS16B} and \cite{CV14} in different directions. 
 In Subsection \ref{subsec:Howardmodel} we describe the Howard's network precisely. 

It is natural to consider a branching extension of the
Howard's Network where branching occurs (with `small' probability depending on the scaling 
coefficient) and ask whether it converges to the Brownian net. Santos et. al. 
\cite{SVZ24} considered a particular variant of branching mechanism and showed that 
under diffusive scaling, the resultant sequence is tight and any sub-sequential limit contains a copy of the Brownian net. It has been conjectured in \cite{SVZ24} that 
the Brownian net is indeed the scaling limit. In this work we consider 
the same variant of branching event, and prove  convergence to the Brownian net.
In addition, we would like to highlight that compared to 
\cite{SVZ24}, we explore the question of scaling limit in a wider 
spectrum of scaling regime (see Remark \ref{rem:ScalingGen}). 

We should mention here that unlike the Brownian web, the universality 
class of the Brownian net has largely remains unexplored. The work by Etheridge et. al. \cite{EFS17} provides the only other example in the basin of attraction of the Brownian net with a coalescing system of {\it dependent} paths where within a compact window, paths interact with each other. In a major breakthrough, Sun et. al. \cite{SSY24} proved 
that the branching coalescing system of 
one-dimensional non-simple random walks, under the optimal $3+\epsilon$
moment finite condition, converge to the Brownian net. Though the random walk paths, considered in \cite{SSY24}, are independent till the time of coalescence, because of the crossing nature of paths 
the convergence question becomes much harder to analyse. 
Our result contributes to the broader 
understanding of the basin of attraction of the Brownian net. 
Our work is based on a careful analysis of the nature of interaction 
between two paths in the branching Howard network and our technique gives an 
independent and efficient self-contained proof for convergence to the Brownian net
(see Remark \ref{rem:Tightness}). 
Our argument is essentially based on the technique used by Etheridge et. al. 
in \cite{EFS17}. However, it should be noted that the network of paths 
considered in \cite{EFS17} are independent as long as distance between 
paths are greater than a fixed constant, whereas the paths in Howard network have 
long rage interactions. We want to highlight that extending the argument 
for paths with long range interactions is non-trivial and to do that we do a closer analysis
of joint evolution of Howard paths. We believe that with suitable modifications, this technique can be applied to some other systems of branching coalescing dependent paths as well. 

In Subsection \ref{subsec:Howardmodel} we describe 
our model and state our main results. In Section \ref{sec:Bweb_Net} we 
provide a brief overview on the Brownian web, Brownian net and we discuss relevant 
convergence conditions. In Section \ref{sec:Howard_Dual} we describe a 
dual process for the Howard network and discuss its behaviour 
under the proposed branching scheme. The main technical part of this paper is the  
analysis of the interaction between two (dual) Howard paths in the branching Howard network,
which is presented in Section \ref{sec:DualPath_Evolution}. 
This analysis allows us to partition the evolution 
time of the joint process of two dual paths into three categories which 
is done in Section \ref{sec:Div_TotalTime}. Finally, 
we prove convergence to the Brownian net in Section \ref{sec:ConvDoubleNet}.       


\subsection{A drainage network model with branching and our main result}
\label{subsec:Howardmodel}

We consider two independent collections:  
\begin{itemize}
\item[(i)] $\{ B_{(x,t)} : (x,t) \in \Z^2 \}$, a collection 
of i.i.d. Bernoulli random variables with success probability 
$p \in (0,1)$, 
\item[(ii)] $\{ R_{(x,t)} : (x,t) \in \Z^2 \}$ a collection of i.i.d. Rademacher 
random variables.
\end{itemize}
Let $V = \{(x,t) \in \Z^2 : B_{(x,t)} = 1\}$ denote the collection of `open' vertices.
For $(x,t) \in \Z \times \Z$ let
\begin{align}
J^{+}_{(x,t)} & := \min\{ k \in \N  : (x + k, t) \in V \}\text{ and }
J^{-}_{(x,t)} := \min\{ k \in \N  : (x - k, t) \in V\} .
\end{align} 
respectively denote the  distance of the nearest open vertex to the 
right and to the left at the same level, i.e., with the same second coordinate $t$.
Let $J$ denote the geometric random variable with parameter $p$ such that 
\begin{align}
\label{def:Geometric}
\P(J = k) = (1-p)^{k-1} p \text{ for all }k \in \N~.
\end{align}
We note that for any $(x,t) \in \Z^2$, the random variables 
$ J^{+}_{(x,t)} $ and $J^{-}_{(x,t)}$ are i.i.d. copies of $J$. 
Further, we observe that for any $(x,t) \in \Z/2 \times \Z$ with $x \notin \Z$,
the random variables, $J^{+}_{(x-1/2,t+1)}$
and $J^{-}_{(x+1/2,t+1)}$, are independent.

For $(x,t) \in \Z^2$, we define the next step $h(x,t) \in V$ as 
\begin{align*}
h(x,t) := 
\begin{cases}
(x, t+1) & \text{ if }(x, t+1) \in V \\
(x + J^+_{(x,t+1)}, t+1 ) & \text{ if }(x, t+1) \notin V \text{ and } 
J^+_{(x,t+1)} < J^-_{(x,t+1)}\\
& \qquad \text{ or if }(x, t+1) \notin V\text{ and }
J^+_{(x,t+1)} = J^-_{(x,t+1)}\text{ with } R_{(x,t)} = +1\\
(x - J^-_{(x,t+1)}, t+1 ) & \text{ if }(x, t+1) \notin V \text{ and } 
J^+_{(x,t+1)} > J^-_{(x,t+1)}\\
& \qquad \text{ or if }(x, t+1) \notin V\text{ and }
J^+_{(x,t+1)} = J^-_{(x,t+1)}\text{ with } R_{(x,t)} = -1.
\end{cases}
\end{align*}

The Howard network is defined as the random graph $G = (V, E)$ where the edge set $E$ is given as 
$E := \{ \langle (x,t), h(x,t)\rangle : (x,t) \in V\}$. Note that each $(x,t) \in V$ has 
exactly one outgoing edge given by $\langle (x,t), h(x,t)\rangle$ and therefore, the resultant 
graph is cycle or loop-free a.s. The next result (Theorem 1 of \cite{GRS04}) shows that 
$G$ is connected a.s. 

\begin{theorem}[Theorem 1 of \cite{GRS04}]
\label{thm:HowardTree}
$G$ is connected and therefore, it consists of a single tree a.s. 
\end{theorem}
Here we describe only the two-dimensional Howard network.
The authors of \cite{GRS04} actually analysed the Howard network for general $d \geq 2$
and established a tree-forest dichotomy depending on dimensions. 

We observe that the random network $G$ can be considered 
as a collection of continuous paths obtained by linearly joining
successive $h(\cdot)$ steps. Set $h^0(x,t) = (x,t)$ and for $k \geq 1$, let $h^k (x,t) = 
h(h^{k-1}(x,t))$ denote the $k$-th step. By linearly joining these successive steps, 
we obtain a continuous path $\pi^{(x,t)}$ starting at $(x,t)$ such that
$\pi^{(x,t)}(t+k) = h^k (x,t)(1)$ for all $k \in \N \cup \{ 0\}$. 
Let 
\begin{align}
\label{def:HowardPath}
\Xi := \{ \pi^{(x,t)} : (x,t) \in V \}
\end{align}
denote the collection of all paths in the Howard network.  
We consider a sequence 
$\{ \epsilon_n : n \geq 1\}$ such that $\epsilon_n \downarrow 0$ as $n \to \infty$.

In what follows, either we will use $(x,t)$ to denote a point in the space-time plane $\R^2$ or we will use bold-faced letters, e.g., $\x, \y, \z$ to denote the same. 
For $\x \in \R^2$ and  for $i = 1,2$ let $\x(i)$ denote the $i$-th coordinate of $\x$. 
Let 
\begin{align}
\label{def:DiffusiveScaling}
S_n(\x) := (\frac{\x(1)\epsilon_n}{\sigma} , \x(2)\epsilon_n^2 ),
\end{align}
denote the $n$-th diffusively scaled version of $\x \in \R^2$ corresponding to the sequence 
$\{ \epsilon_n : n \geq 1\}$ with the normalization constant 
\begin{align}
\label{def:Diff_Coeff}
\sigma = \sigma(p) = \frac{(1-p)(2 - 2p +p^2)}{p^2(2-p)^2}~.
\end{align}
For any subset $A \subset \R^2$ with a slight abuse of notation 
$S_n(A)$ denotes the $n$-th diffusively scaled version of the set $A$ given by 
$ S_n(A) = \{ (\frac{\x(1)\epsilon_n}{\sigma} , \x(2)\epsilon_n^2) : \x \in A \} $. 
For any path $\pi \in \Pi$, where $\Pi$ denotes the space of continuous paths with all possible starting times and it will be defined precisely in Section \ref{subsec:Bweb_Dual},
 let $\sigma_\pi$ denote its starting time. 
We identify $\pi$ with it's trajectory $\{(\pi(t), t) : t \geq \sigma_\pi\}$ 
as a subset of $\R^2$ and accordingly the scaled version $S_n(\pi)$
is defined. For $K \subset \Pi$, the notation $S_n(K) := \{ S_n(\pi) : \pi \in K\}$ denotes the 
$n$-th scaled version of paths in $K$. 

As mentioned before, Colletti et. al. \cite{CFD09} showed that under diffusive 
scaling as collection of paths, scaled network converges in distribution to the 
Brownian web in an appropriate topological setting. Intuitively, the Brownian web can
be described as the collection of coalescing Brownian motions starting from everywhere in 
the space-time plane $\R^2$. In Subsection \ref{subsec:Bweb_Dual} we present a brief
introduction on the Brownian web and discuss the topology of convergence.    

Now we are going to describe the branching Howard network that was studied in \cite{SVZ24}. We define $V_{\text{sp}} \subset V$ as 
$$
V_{\text{sp}} := \{ (x,t) \in V \text{ such that } (x,t+1) \notin V \text{ and }
J^+_{(x,t+1)} = J^-_{(x,t+1)} \}.  
$$
In other words, $V_{\text{sp}}$ denotes the set of vertices with two equally nearest open vertices at the next level. For $(x,t) \in V_{\text{sp}}$ we resolve the tie using the Rademacher random variable $R_{(x,t)}$ to obtain the unique outgoing edge 
$\langle (x,t), h(x,t)\rangle$. 
The authors of \cite{SVZ24} considered a specific type of branching which occurs only at 
vertices in $V_{\text{sp}}$. The same branching mechanism is considered here. 
We consider a sequence 
$\{ \beta_n : n \geq 1\} \subset (0,1)$ such that
\begin{align}
\label{eq:Limit_b}
\beta_n/\epsilon_n \to b \in [0,\infty) \text{ as }n \to \infty~.
\end{align}
We assume that branching occurs only at vertices in $V_{\text{sp}}$ with probability $\beta_n$ independent of anything else.
In case branching occurs at $(x,t) \in V_{\text{sp}}$, then 
$(x,t)$ will have two outgoing edges connecting $(x,t)$ to 
$(x + J^+_{(x,t+1)},t+1)$ and  $(x - J^-_{(x,t+1)},t+1)$
instead of one. Below we describe this model precisely.

We consider a collection $\{ \text{Br}^n_{(x,t)} : (x,t) \in \Z^2 \}$ of i.i.d. Bernoulli
r.v.s with success probability $\beta_n$, independent of the collection 
$\{ (B_{(x,t)}, R_{(x,t)}) : (x,t) \in \Z^2 \}$ and we denote the collection of i.i.d. random vectors as 
\begin{align}
\label{def:LambdaRandomVector}
\{ \Lambda_{(x,t)} := (B_{(x,t)}, R_{(x,t)}, \text{Br}^n_{(x,t)}) : (x,t) \in \Z^2 \}.
\end{align} 
To describe the branching event at $(x,t) \in V$, we define two steps 
$h_{n,(r)}(x,t)$ and $h_{n,(l)}(x,t)$
as follows: 
\begin{align*}
(h_{n,(l)}(x,t) , h_{n,(r)}(x,t)) := 
\begin{cases}
(h(x,t), h(x,t)) & \text{ if }(x,t) \in V \setminus V_{\text{sp}}\\
(h(x,t), h(x,t)) & \text{ if }(x,t) \in V_{\text{sp}}\text{ and }\text{Br}^n_{(x,t)} = 0\\
\bigl( (x - J^{-}_{(x,t+1)},t+1) , (x + J^{+}_{(x,t+1)},t+1) \bigr ) & 
\text{ if }(x,t) \in V_{\text{sp}}\text{ and }\text{Br}^n_{(x,t)} = 1.
\end{cases}
\end{align*} 
In other words, branching is visible at $(x,t) \in V$ if and only if $(x,t) \in V_{\text{sp}}$ with 
$\text{Br}^n_{(x,t)} = 1$ and in the case of branching, the rightmost (leftmost) step 
(corresponding to the branching parameter $\beta_n$) is given by $h_{n,(r)}(x,t)$ 
($h_{n,(l)}(x,t)$). 

The collection of all outgoing edges in the branching Howard model 
with branching probability $\beta_n$ is defined as 
\begin{align*}
E^n := \{ \langle (x,t), h_{n, (r)}(x,t)\rangle , 
\langle (x,t), h_{n, (l)}(x,t)\rangle : (x,t) \in V\}.
\end{align*}
The collection of {\it all} paths obtained in the branching Howard model (with
branching parameter $\beta_n$) by following all the outgoing edges is defined as 
\begin{align*}
\NN^n := \Bigl \{ & \pi \in \Pi : (\pi(\sigma_\pi),\sigma_\pi) = (x,t) \text{ for some }
(x,t) \in V \text{ and for all }k \in [t, \infty) \cap \Z \text{ we have }    \\
& (\pi(k), k) \in  \{ h^{m_1}_{n, (l)}(h^{n_1}_{n, (r)}(\cdots (h^{n_j}_{n, (r)}(x,t)))) : 
m_1, n_1, \cdots , m_j,  n_j \geq 0 \text{ with }\sum_{i=1}^j(m_i + n_i) = k - t\}\\
&   \text{ and for }s \in (k-1,k) \text{ we linearly join the vertices }
(\pi(k-1), k-1) \text{ and }(\pi(k), k) \Bigr \}.
\end{align*}
The next result proves the conjecture posed in \cite{SVZ24}.
Let $\overline{S_n(\NN^n)}$ denote the closure of $\NN^n$ taken in an 
appropriate topology which will be explained in Subsection \ref{subsec:Bweb_Dual}. 
\begin{theorem}
\label{thm:Net_Conv}
We recall that $b$ is the limit of the sequence $\{ \beta_n/\epsilon_n : n \geq 1\}$ (see (\ref{eq:Limit_b})). For any $b \in [0, \infty)$, the sequence $\overline{S_n(\NN^n)}$ 
converges in distribution to $\mathcal{Z}_b$
in an appropriate topological setting. The distribution of the limiting 
random variable $\mathcal{Z}_b$ depends on $b$ such that  
$\mathcal{Z}_0$ is distributed as the standard Brownian web $\WW$ whereas for any $b > 0$, 
$\mathcal{Z}_b$ is distributed as the Brownian net $\NN_{\bar{b}}$ with the drift parameter 
\begin{align}
\label{def:Drift_Param}
\bar{b} = b(p) = b\frac{(1-p)}{(2-p)^2}.
\end{align}
\end{theorem}
\begin{remark}
\label{rem:ScalingGen}
We have mentioned that Theorem \ref{thm:Net_Conv} proves the conjecture 
posed in \cite{SVZ24}. It is important to mention that
the authors of \cite{SVZ24} considered the scaling limit question
in a more restricted regime, which includes the case $\beta_n = \epsilon_n$,
and proved that the sequence $\{ \overline{S_n(\NN^{n})} : n \geq 1\}$ 
is tight and any sub-sequential limit of it contains a copy of the Brownian net
 with drift parameter $\frac{(1-p)}{(2-p)^2}$. We will 
 state and prove a stronger version in Theorem \ref{thm:DoubleNet_Conv}.
\end{remark}
In the next section we provide a brief introduction on the Brownian web and Brownian net.

\section{Brownian web and net}
\label{sec:Bweb_Net}
\subsection{Brownian web and it's dual}
\label{subsec:Bweb_Dual}
We mentioned earlier that heuristically, the Brownian web can be described as the collection 
of coalescing Brownian motions starting from everywhere on $\R^2$. In order to make it 
rigorous, we need to recall the relevant topological details from \cite{FINR04}.
Let $\R^{2}_c$ denote the completion of the space time plane $\R^2$ with
respect to the metric
\begin{equation*}
\rho((x_1,t_1),(x_2,t_2))=|\tanh(t_1)-\tanh(t_2)|\vee \Bigl| \frac{\tanh(x_1)}{1+|t_1|}
-\frac{\tanh(x_2)}{1+|t_2|} \Bigr|.
\end{equation*}
As a topological space $\R^{2}_c$ can be identified with the
continuous image of $[-\infty,\infty]^2$ under a map that identifies the line
$[-\infty,\infty]\times\{\infty\}$ with the point $(\ast,\infty)$, and the line
$[-\infty,\infty]\times\{-\infty\}$ with the point $(\ast,-\infty)$.
A path $\pi$ in $\R^{2}_c$ with starting time $\sigma_{\pi}\in [-\infty,\infty]$
is a mapping $\pi :[\sigma_{\pi},\infty]\rightarrow [-\infty,\infty]$ such that
$\pi(\infty)= \pi(-\infty)= \ast$ and $t\rightarrow (\pi(t),t)$ is a continuous
map from $[\sigma_{\pi},\infty]$ to $(\R^{2}_c,\rho)$.
We then define $\Pi$ to be the space of all paths in $\R^{2}_c$ with all possible starting times in $[-\infty,\infty]$ equipped with the following metric, 
\begin{equation*}
d_{\Pi} (\pi_1,\pi_2)= |\tanh(\sigma_{\pi_1})-\tanh(\sigma_{\pi_2})|\vee\sup_{t\geq
\sigma_{\pi_1}\wedge
\sigma_{\pi_2}} \Bigl|\frac{\tanh(\pi_1(t\vee\sigma_{\pi_1}))}{1+|t|}-\frac{
\tanh(\pi_2(t\vee\sigma_{\pi_2}))}{1+|t|}\Bigr|,
\end{equation*}
for $\pi_1,\pi_2\in \Pi$. This metric 
makes $\Pi$ a complete, separable metric space. Convergence in this
metric can be described as locally uniform convergence of paths as
well as convergence of starting times.
Let ${\cal H}$ be the space of compact subsets of $(\Pi,d_{\Pi})$ equipped with
the Hausdorff metric $d_{{\cal H}}$ given by,
\begin{equation*}
d_{{\cal H}}(K_1,K_2)= \sup_{\pi_1 \in K_1} \inf_{\pi_2 \in
K_2}d_{ \Pi} (\pi_1,\pi_2)\vee
\sup_{\pi_2 \in K_2} \inf_{\pi_1 \in K_1} d_{\Pi} (\pi_1,\pi_2).
\end{equation*}
The space $({\cal H},d_{{\cal H}})$ is a complete separable metric space. Let
$B_{{\cal H}}$ be the Borel  $\sigma-$algebra on the metric space $({\cal H},d_{{\cal H}})$.
The Brownian web ${\mathcal W}$ is defined and characterized as an 
$({\cal H}, B_{{\cal H}})$ valued random variable by the following result of Fontes et. al. \cite{FINR04}
which shows that the Brownian web is almost surely determined by countably many coalescing Brownian motions.

\begin{theorem}[Theorem 2.1 of \cite{FINR04}]
\label{theorem:Bwebcharacterisation}
There exists an $({\mathcal H}, {\mathcal B}_{{\mathcal H}})$-valued random variable
${\mathcal W}$ whose distribution is uniquely determined by
the following properties:
\begin{itemize}
\item[$(a)$] from any deterministic point $\x\in\R^2$, there is  almost surely a unique path $\pi^{\x}\in {\mathcal W}$  starting from $\x$;
\item[$(b)$] for a finite set of deterministic points $\x_1,\dotsc, \x_k \in \R^2$, the collection $(\pi^{\x_1},\dotsc,\pi^{\x_k})$ is distributed as coalescing 
Brownian motions starting from $\x_1,\dotsc,\x_k$;
\item[$(c)$] for any countable deterministic dense set ${\mathcal D}$ of $\R^2$, ${\mathcal W}$ is the closure of $\{\pi^{\x}: \x\in {\mathcal D} \}$ in $(\Pi, d_{\Pi})$  almost surely.
\end{itemize}
\end{theorem}

Let $\Xi \subset \Pi$. For $t > 0$ and $t_0, 
c, d \in  \R$ with $c < d$, consider the counting random variable 
$\eta_{\Xi}(t_0, t; c, d)$ defined as 
\begin{align}
\label{def:EtaPtSet}   
\eta_{\Xi}(t_0, t; c, d) := \#\{ \pi(t_0 + t) 
: \pi \in \Xi, \sigma_\pi \leq t_0 
\text{ and }\pi(t_0) \in [c, d] \}, 
\end{align}
 i.e., it considers paths in $\Xi$, those start before time $t_0$ 
 and pass through $[c, d]$ at time $t_0$ 
and it counts the number of different positions  occupied by these paths at time $t_0 + t$. 
In Theorem 2.2 of \cite{FINR04}, Fontes et al. provided
the first set of convergence criteria to the Brownian web applicable for models with non-crossing paths. The survey article by Schertzer et al. \cite{SSS19} provides an excellent 
comprehensive overview on this topic. 
For $\x \in \R^2$ there exists  $\{ \x^n : n \geq 1 \} \subset V$  such that 
almost surely we have $S_n (\x^n) \to \x$ as $n \to \infty$.
\begin{theorem}[Theorem 2.2 of \cite{FINR04}]
\label{thm:BwebConvFINR} 
Let $\{ \Xi_n : n \in \N \}$ be a sequence of $(\mathcal{H}, \mathcal{B}_{\mathcal{H}})$ valued random variables with non-crossing paths. Assume that the following conditions hold:
\begin{itemize}
    \item[($I_1$)] Fix a deterministic countable dense set $\mathcal{D}$ of $\R^2$. For each $\x \in  \mathcal{D}$, there exists $\pi^{\x^n}_n  \in \Xi_n$ such that for any finite set of points $\x_1, \cdots , \x_k \in \mathcal{D}$, as $n \to \infty$, we have $(\pi^{{\x^n_1}}_n , \cdots , \pi^{{\x^n_k}}_n )$ converges
in distribution to $(W^{\x_1}, W^{\x_2}, \cdots ,W^{\x_k})$, where $(W^{\x_1}, W^{\x_2}, \cdots ,W^{\x_k})$ denotes coalescing Brownian motions starting from points $\x_1, \cdots, \x_k$.
    \item[($B_1$)] For all $t > 0$, $ \limsup_{n\to \infty} \sup_{(a,t_0) \in \R^2} 
    \P( \eta_{\Xi_n}(t_0, t; a, a + \epsilon ) \geq 2) \to 0$ as $\epsilon \downarrow 0$.
    \item[($B_2$)] For all $t > 0$, $\frac{1}{\epsilon} \limsup_{n\to \infty} \sup_{(a,t_0) \in \R^2} \P( 
    \eta_{\Xi_n}(t_0, t; a, a + \epsilon ) \geq 3) \to 0$ as $\epsilon \downarrow 0$.
\end{itemize}
Then as a sequence of $(\mathcal{H}, \mathcal{B}_{\mathcal{H}})$ valued random variables, 
$\{ \Xi_n : n \in \N \}$ converges in distribution to the Brownian web $\WW$. 
\end{theorem}


We note that for the Howard network, the paths are a.s. non-crossing. 
Next we state an alternate set of conditions for convergence to the Brownian web derived in \cite{SSS19}. This set of conditions is applicable only for non-crossing path models. 
We need to introduce the dual Brownian web $\widehat{\WW}$.
As in the case of forward paths, one can consider a similar metric space of collection of backward paths 
denoted by $(\widehat{\Pi}, d_{\widehat{\Pi}})$. 
Let $(\widehat{{\mathcal H}}, d_{\widehat{{\mathcal H}}})$
denote the corresponding Polish space of compact collections of backward paths with the induced Hausdorff metric. The Brownian web and its dual, denoted 
by $({\mathcal W},\widehat{{\mathcal W}})$, is a $({\mathcal H}\times \widehat{{\mathcal H}}, {\mathcal B}_{{\mathcal H}\times \widehat{{\mathcal H}}})$-valued random variable, which is referred as the double Brownian web, satisfying the following:
\begin{itemize}
\item[$(i)$] The dual  Brownian web $\widehat{{\mathcal W}}$ is distributed as the Brownian web  rotated $180^0$ about the origin;
\item[$(ii)$] ${\mathcal W}$ and $\widehat{{\mathcal W}}$ a.s. 
uniquely determine each other
in the sense that the paths of ${\cal W}$ a.s.~do not cross with (backward) paths in 
$\widehat{{\cal W}}$ and it has been observed in \cite{SSS19} that, $\WW$, which 
contains a collection of coalescing non-crossing Brownian paths starting from a dense set, 
almost surely uniquely determines the backward collection of paths.    
\item[$(iii)$] The interaction between the paths in ${\mathcal W}$ and $\widehat{\mathcal W}$ is that of Skorohod reflection (see \cite{STW00}).
\end{itemize}
Using $({\mathcal W},\widehat{{\mathcal W}})$, Sun and Swart in \cite{SS08}
 provided a `wedge' characterization of $\WW$.
Consider $\widehat{\pi_1}, \widehat{\pi_2} \in \widehat{\WW}$ with $s = \sigma_{\widehat{\pi_1}} \wedge \sigma_{\widehat{\pi_2}}$ 
and $\widehat{\pi_1}(s) < \widehat{\pi_2}(s)$. 
Let $T$ denote the first intersection time of  
$\widehat{\pi}_1, \widehat{\pi}_2$ defined as
$T := \inf\{ t < s : \widehat{\pi}_1(t) = \widehat{\pi}_1(t)\}$. 
The wedge between $\widehat{\pi}_1, \widehat{\pi}_2$
is denoted by $W(\widehat{\pi}_1, \widehat{\pi}_1)$ and defined as the open set
$$
W(\widehat{\pi}_1, \widehat{\pi}_2) := \{ (u,t) : T < t < s , 
\widehat{\pi}_1(t) < u < \widehat{\pi}_2(t) \}~.
$$
We say that a forward path $\pi \in \Pi$, enters a wedge 
$W(\widehat{\pi}_1, \widehat{\pi}_2)$
from outside if there  exist $\sigma_\pi < t_1 < t_2$ such that   
$$
\pi(t_1) \notin \overline{W}(\widehat{\pi}_1, \widehat{\pi}_2) \text{ and }
\pi(t_2) \in W(\widehat{\pi}_1, \widehat{\pi}_2),
$$
where $\overline{W}(\widehat{\pi}_1, \widehat{\pi}_2)$ denotes closure of 
$W(\widehat{\pi}_1, \widehat{\pi}_2)$.
The following `wedge' characterisation theorem of $\WW$ is taken from Theorem 1.9 of \cite{SS08}.

\begin{theorem}[Theorem 1.9 of \cite{SS08}]
\label{thm:BwebWedge}
For $({\mathcal W},\widehat{{\mathcal W}})$ almost surely we have 
\begin{align*}
\WW := \{\pi \in \Pi : \pi \text{ does not enter any wedge in }\widehat{\WW}\text{ from 
outside}\}.
\end{align*}
\end{theorem}

Using this `wedge' characterisation of $\WW$, in \cite{SSS19} Sun et. al. developed Theorem \ref{thm:ConvBweb_Noncrossing} for convergence to $\WW$, which is applicable for non-crossing path models only. 
Before stating Theorem \ref{thm:ConvBweb_Noncrossing} we make the following remark. 

\begin{remark}
\label{rem:Dual_Noncrossing}
We observe that for a sequence of $( \mathcal{H}, \mathcal{B}_{\mathcal{H}})$-valued random variables 
$\{ \Xi_n : n \geq 1\}$ with non-crossing paths satisfying condition $I_1$ (see Theorem \ref{thm:BwebConvFINR}), the sequence $\{ \Xi_n : n \geq 1\}$ is tight and any sub-sequential limit 
$\mathcal{Z}$ of  $\{ \Xi_n : n \geq 1\}$ contains a copy of $\WW$. 
Further, as observed in \cite{SSS19}, 
there exists a dual sequence $\{ \widehat{\Xi}_n : n \geq 1\}$
of non-crossing backward paths such that the following conditions hold:
\begin{itemize}
\item[(i)] Any backward/dual path $\widehat{\pi} \in \widehat{\Xi}_n$ cannot cross a
forward path in $\Xi_n$.
\item[(ii)] As $n \to \infty$, the set of starting points of dual paths in $\widehat{\Xi}_n$, given by $\{(\widehat{\pi}(\sigma_{\widehat{\pi}}),\sigma_{\widehat{\pi}}) : \pi \in \widehat{\Xi}_n\}$, becomes 
dense in $\R^2$.
\item[(iii)] The sequence $\{  (\Xi_n, \widehat{\Xi}_n) : n \geq 1 \}$ is jointly tight and any 
sub-sequential limit $(\mathcal{Z},\widehat{\mathcal{Z}})$ contains a copy of the 
double Brownian web $(\WW, \widehat{\WW})$.   
\end{itemize} 
\end{remark}
For $K \subset \Pi, \widehat{K} \subset \widehat{\Pi}$ and $A\subset \R^2$ let 
\begin{align}
\label{def:PathsStartingPoints}
K^A  & := \{ \pi \in K : (\pi(\sigma_\pi),\sigma_\pi) \in A\} \text{ and } \nonumber \\
\widehat{K}^A & := \{ \widehat{\pi} \in \widehat{K} : 
(\widehat{\pi}(\sigma_{\widehat{\pi}}),\sigma_{\widehat{\pi}}) \in A\}
\end{align}
denote the collection of paths in $K$ and $\widehat{K}$ respectively 
with starting points in $A$.
We state Theorem \ref{thm:ConvBweb_Noncrossing} from \cite{SSS19} below
which would be useful for us.
\begin{theorem}[Theorem 6.6 of \cite{SSS19}]
\label{thm:ConvBweb_Noncrossing}
For a sequence of $( \mathcal{H}, \mathcal{B}_{\mathcal{H}})$-valued random variables 
$\{ \Xi_n : n \geq 1\}$ with non-crossing paths satisfying condition $I_1$, we consider 
a dual sequence $\{ \widehat{\Xi}_n : n \geq 1\}$ satisfying properties as mentioned 
in Remark \ref{rem:Dual_Noncrossing}. For any sub-sequential limit 
$(\mathcal{Z},\widehat{\mathcal{Z}})$ of $\{ (\Xi_n, \widehat{\Xi}_n) : n \geq 1\}$
we also have that no path in $\mathcal{Z}^{\Q^2}$ (defined as in (\ref{def:PathsStartingPoints})) enters any 
wedge in $\widehat{\mathcal{Z}}^{\Q^2}$ from {\it outside}
 and vice versa, i.e., no path in $\widehat{\mathcal{Z}}^{\Q^2}$ enters any 
wedge in $\mathcal{Z}^{\Q^2}$ from {\it outside}. 
Then $(\mathcal{Z}, \widehat{\mathcal{Z}})$ has the same distribution as the double Brownian web $(\WW, \widehat{\WW})$.
\end{theorem}



\label{subsec:Bweb_Dual}

\subsection{Brownian net}
\label{subsec:BrownianNet}

In this section we describe the Brownian net which generalizes 
the Brownian web by allowing paths to branch. 
The Brownian net was introduced by Sun and Swart in \cite{SS08}.
Brownian net  with drift parameter $d \in (0, \infty)$ is denoted by $\NN_d$ and described below. In $\NN_d$ as a result of branching, 
for any deterministic $(x,t) \in \R^2$ there are multitudes of 
paths starting from $(x,t)$. Among the set of paths starting from $(x,t)$ in $\NN_d$, 
the rightmost $r^{(x,t)}$ (resp. the leftmost $l^{(x,t)}$) path starting from $(x,t)$ is a.s. uniquely defined. 
The collection of the leftmost and the rightmost paths in the Brownian net 
gives rise to what is called the left-right Brownian web. 
We first describe the left-right Brownian Web. 

For a collection of space-time points $\z_1, \cdots , \z_k, 
\z^\prime_{1},  \cdots , \z^\prime_{k^\prime}$ in $\R^2$, the joint distribution of leftmost and rightmost paths 
$\bigl ( l^{\z_1}, \cdots , l^{\z_k}, r^{\z^\prime_{1}}, 
\cdots , r^{\z^\prime_{k^\prime}} \bigr )$ with drift $d \in (0, \infty)$ 
is given as follows:
\begin{itemize}
\item[(a)] The collection $\bigl ( l^{\z_1}, \cdots , l^{\z_k} \bigr )$
is distributed as independent coalescing Brownian motions starting from 
$\z_1, \cdots , \z_{k}$ with drift $-d$. Similarly, 
the collection $\bigl ( r^{\z^\prime_{1}}, \cdots , 
r^{\z^\prime_{k^\prime}} \bigr )$ is distributed as  
coalescing Brownian motions starting from 
$\z^\prime_{1}, \cdots , \z^\prime_{k^\prime}$ with drift $+d$. 

\item[(b)] All these paths evolve independently when they are disjoint. 

\item[(c)] When a pair of left path and right path intersect, they follow a `sticky' 
interaction. Because of Item (b) together with the strong Markov property, 
it suffices to characterize the interaction between
one pair of leftmost and rightmost paths staring from the same point. W.l.o.g. 
we take $(L_t, R_t)_{t \geq 0} = (l^{(0,0)}(t), r^{(0.0)}(t))$ and $(L_t, R_t)_{t \geq 0}$ 
solves the following system of ODE's:
\begin{align}
\label{eq:LR_SDE}
d L_t & = \mathbf{1}_{\{ L_t \neq R_t\}} d B^l_t + \mathbf{1}_{\{ L_t = R_t\}} d B^c_t - d dt ,\nonumber\\
d R_t & = \mathbf{1}_{\{ L_t \neq R_t\}} d B^r_t + \mathbf{1}_{\{ L_t = R_t\}} d B^c_t + d dt , 
\end{align}    
where $B^l,B^r, B^c$ are independent standard Brownian motions, and $L_t$ and $R_t$ are
subject to the constraint that $L_t \leq R_t$ for all $t  \geq 0$.
$(L_t, R_t)_{t \geq 0}$ represents left-right Brownian motions 
starting from origin with drift $d$.

As observed in \cite{SS08}, (\ref{eq:LR_SDE}) can be rewritten into a different 
sets of equations shown below, which has a pathwise
unique solution.
\begin{align}
\label{eq:LR_SDE_new}
 d L_t & =  d B^l_{I_t} +  d B^c_{C_t} - d dt ,\nonumber\\
 d R_t & =  d B^r_{I_t} + d B^c_{C_t} + d dt ,\nonumber\\
 t & = I_t + C_t , \nonumber\\
 0 & = \int_0^t \mathbf{1}_{\{ L_t <  R_t\}} dC_t, \nonumber\\
 C_t&= \int_0^t \mathbf{1}_{\{ L_s = R_s\}} ds
\end{align}

where $B_l$, $B_r$ and $B_c$ are independent standard one 
dimensional Brownian motions.
It has been proved in \cite{SS08} that there is a one-to-one correspondence
in law between weak solution of (\ref{eq:LR_SDE}) and weak solution of 
(\ref{eq:LR_SDE_new}).

\item[(d)] A straightforward extension of (\ref{eq:LR_SDE}) is sufficient to specify the joint distribution of any finite collection of left-right paths, known as left-right coalescing Brownian motions. However, that is a rather laborious route. Instead, we will 
follow Section 2.2 of \cite{SS08} to present a rigorous construction of left-right 
coalescing Brownian paths starting from $\z_i = (x_i, t_i), \z^\prime_i = 
(x^\prime_i, t^\prime_i) \in \R^2$  for $1 \le i \le k$ and $ 1 \le i \le k^\prime$. 
The times $t_1, \cdots , t_k$, $t^\prime_1, \cdots , t^\prime_{k^\prime}$ 
divide $\R$ into a finite number of intervals and it suffices to specify a Markov process that describes the time evolution of the left-right coalescing Brownian 
motions during each such interval. As explained, the goal is to construct a 
Markov process 
$(L_{1,t}, \cdots , L_{k,t} ; R_{1,t} , \cdots ,R_{k^\prime,t} )_{t \geq 0}$
in $\R^{k + k^\prime}$ such that $(L_{1,t}, \cdots , L_{k,t})_{t \geq 0}$ and 
$(R_{1,t}, \cdots , R_{k^\prime,t})_{t \geq 0}$ are each distributed as coalescing
Brownian motions with drift $- d$ and $ + d$ respectively, and the interaction
between paths in $(L_1 , \cdots , L_k )$ and $(R_1 , \cdots , R_{k^\prime} )$ 
is that of the left-right SDE (as given in (\ref{eq:LR_SDE})). At the time $t := \max\{t_i : 1 \leq i \leq k \}\vee \max\{t_i : 1 \leq i \leq k^\prime\}$ we order the  paths from 
left to right w.r.t. their values at time $t$. W.l.o.g. we assume that our system is initially ordered as `$LRRLRLRLLLRRLR$', i.e., the first path follows a left 
Brownian motion whereas the second path follows a right Brownian motion. We 
partition this collection as $\{ LR \}\{ R \}\{ LR\}\{ LR \} \{L \}\{ L \}
\{ LR \}\{ R \}\{ LR\}$, letting all pairs of a left Brownian motion
followed by a right Brownian motion constitute a partition element with two
members, and putting all remaining Brownian motions into partition elements with
one member. We let the partition elements evolve independently,
where partition elements containing two members evolve
according to the left-right SDE (\ref{eq:LR_SDE}), till the first time 
\begin{itemize}
\item[(i)] either two left (right) Brownian motions intersect 
and coalesce,
\item[(ii)] or a right Brownian motion hits a left Brownian motion from the left. As the right Brownian motion has to continue to the right of the left Brownian motion,
this results in a crossing.
\end{itemize}
By construction, all the above random times (coalescing or crossing) are hitting times.
At the first coalescing or crossing time,  the Brownian motions, those have hit each
other, are respectively coalesced or crossed. Post that we repartition the system and continue the system till the next coalescing or crossing event. As observed in \cite{SS08}, 
we can have at most finitely many such iterations which eventually leads to a single pair
$(L,R)$ provided both $k , k^\prime \geq 1$. This construction uniquely defines 
the system of left-right coalescing Brownian motions 
$l^{\z_1}, \cdots , l^{\z_k}$, $r^{\z^\prime_1}, \cdots , r^{\z^\prime_{k^\prime}}$ 
and the consistency property allows us to define the system 
of left-right coalescing Brownian motions starting from a 
countable deterministic dense set in $\R^2$.
\end{itemize}

It is useful to observe that for $d = 0$, the ODE in Equation 
(\ref{eq:LR_SDE}) simply gives coalescing Brownian motions ({\it without any drift}).  
The above properties uniquely determine the joint law of  the collection
$$
\bigl ( l^{\z_1}, \cdots , l^{\z_k}, r^{\z^\prime_1}, 
\cdots , r^{\z^\prime_{k^\prime}} \bigr ),
$$
 which are called left-right coalescing Brownian motions starting from $\z_1, \cdots , \z_k
 , \z^\prime_1, \cdots , \z^\prime_{k^\prime}$ with drift parameter $d$.
The theorem (Theorem 1.5 of \cite{SS08}) below gives a characterization of the
left-right Brownian web with drift parameter $d > 0$, 
denoted as $(\WW^l_d, \WW^r_d)$.
\begin{theorem}[Theorem 1.5 of \cite{SS08}]
\label{thm:LRBweb}
There exists a $\mathcal{H}^2$-valued random variable $(\WW^l_d, \WW^r_d)$, the  
left-right Brownian web with drift parameter $d$, whose distribution 
is uniquely determined by the following properties:
\begin{itemize}
\item[(a)] Each deterministic $\z \in \R^2$, has unique path $l^\z \in \WW^l_d$ 
 and $r^\z \in \WW^r_d$, both starting from $\z$ a.s.
 \item[(b)] For any finite deterministic set of points 
 $\z_1, \cdots , \z_k, \z^\prime_1, \cdots , 
 \z^\prime_{k^\prime} \in \R^2$, the collection of paths 
 $$
 \bigl ( l^{\z_1}, \cdots , l^{\z_k}, r^{\z^\prime_1}, \cdots , 
 r^{\z^\prime_{k^\prime}} \bigr ) 
 $$ 
 is distributed as a collection of left-right coalescing Brownian motions with drift parameter $d$ (see (\ref{eq:LR_SDE})).
 \item[(c)] For any deterministic countable dense sets $\mathcal{D}^l,\mathcal{D}^r \subset \R^2$,
 we have 
 $$
 \WW^l_d = \overline{\{ l^\z : \z \in \mathcal{D}^l\}} \text{ and }
 \WW^r_d = \overline{\{ r^\z : \z \in \mathcal{D}^r\}} \text{ a.s.}
 $$
\end{itemize}
\end{theorem}

Following \cite{SS08}, we are now going to construct the Brownian net $\NN_d$ 
with drift parameter $d > 0$ using the left-right Brownian 
web $(\WW^l_d, \WW^r_d)$. We need to first explain 
the hopping operation on a subset of paths $K \subset \Pi$. 
For $\pi_1, \pi_2 \in \Pi$ we define the set of intersection times between $\pi_1, \pi_2$ 
as 
\begin{align}
A_{\pi_1, \pi_2} := \{ s > \sigma_{\pi_1}\vee \sigma_{\pi_2} : \pi_1(s) = \pi_2(s) \text{ and }
\pi_1(s-) \neq \pi_2(s-) \}, 
\end{align}
and the first intersection time between $\pi_1, \pi_2$ is  
defined as $T_{\pi_1, \pi_2} := \inf (A_{\pi_1, \pi_2})$.
Note that the set $A_{\pi_1, \pi_2}$ can be empty and consequently, $T_{\pi_1, \pi_2}$ can be $+\infty$ 
as well. For $\widehat{\pi}_1, \widehat{\pi}_2 \in \widehat{\Pi}$ 
the set of intersection times and the first 
intersection time are respectively denoted by $A_{\widehat{\pi}_1, \widehat{\pi}_2}$ and $T_{\widehat{\pi}_1, \widehat{\pi}_2}$. 

For any collection of paths $K \subset \Pi$, we let 
${\mathcal{H}}_{\text{int}}(K)$ denote the smallest set
of paths containing $K$ that is closed under hopping at intersection times, that is, 
${\mathcal{H}}_{\text{int}}(K)$ is the set of all paths $\pi \in \Pi$ of the form 
\begin{align}
\label{def:Hop}
\pi = \cup_{k =1}^m \{(\pi_k(t),t) : t \in [t_{k-1}, t_k]\},
\end{align}
where $\pi_1, \cdots , \pi_m \in K$, $\sigma_{\pi_1} = t_0 
< \cdots < t_m = \infty$, and $t_k$ is an intersection time
of $\pi_k$ and $\pi_{k+1}$ for each $k = 1, \cdots , m - 1$.
In other words, ${\mathcal{H}}_{\text{int}}(K)$ contains all paths obtained through 
concatenation of paths in $K$ at intersection times. 
We are now ready to describe the `hopping' construction of $\NN_d$
using $(\WW^l_d, \WW^r_d)$ (Theorem 1.3 of \cite{SS08}). 

\begin{theorem}[Theorem 1.3 of \cite{SS08}]
Consider $(\WW^l_d, \WW^r_d)$ and for any deterministic $\z \in \R^2$ let $l^\z$ 
(resp. $r^\z$) denote the a.s. unique path starting from $\z$ in $\WW^l_d$ (resp. $\WW^r_d$). 
Let $\mathcal{D}^l,\mathcal{D}^r \subset \R^2$ are any two deterministic countable dense sets. 
There exists a $\mathcal{H}$-valued random variable, the  
Brownian net with parameter $d > 0$, denoted by $\NN_d$ such that we have 
 $$
 \NN_d = \overline{ {\mathcal{H}}_{\text{int}} \bigl( \{ l^\z \in \WW^l_d : \z \in \mathcal{D}^l\}\cup 
 \{ r^\z \in \WW^r_d : \z \in \mathcal{D}^r\} \bigr )} \text{  a.s.} 
$$
\end{theorem}

\begin{remark}
The above theorem represents  the `hopping' construction of the Brownian net. 
There are two other constructions of the Brownian net, viz., the `wedge' 
construction and the `mesh' construction. 
In this paper we will describe only the `wedge' construction
by \cite{SS08} as it will be useful for us.    
\end{remark}

In order to simplify the notation, we denote $\WW^l_1, \WW^r_1, \NN_1$ simply as 
$\WW^l, \WW^r, \NN$ and we describe the wedge construction for the choice $d = 1$ only. 
For general $d > 0$ the construction is similar.
We earlier mentioned that the Brownian web $\mathcal{W}$ a.s. uniquely determines its 
dual $\widehat{\WW}$ and vice versa. As the right Brownian web $\WW^r$ forms 
a non-crossing coalescing path family containing paths starting from a dense set, 
it a.s. uniquely determines its dual $\widehat{\WW}^r$
such that the dual paths in $\widehat{\WW}^r$ do not cross forward paths in $\WW^r$. 
Similarly, $\WW^l$ a.s. uniquely determines $\widehat{\WW}^l$.
As observed in \cite{SS08}, the dual left-right Brownian web $(\widehat{\WW}^l, \widehat{\WW}^r)$ is equally distributed with $-(\WW^l,\WW^r)$, 
the left-right Brownian web rotated by $180^0$ 
around the origin. 

For any $\widehat{r} \in \widehat{\WW}^r $ and 
$\widehat{l} \in \widehat{\WW}^l $ with 
$\widehat{r}(s) < \widehat{l}(s)$ at the time $s := \sigma_{\widehat{r}}
\wedge \sigma_{\widehat{l}}$, let $T$ denote the first intersection time 
$T_{\widehat{r} , \widehat{l}}$. Note that $T$ could be $- \infty$. The open set 
between the excursion of $\widehat{r}$ and $\widehat{l}$ 
during the time duration $(T < u < s)$ with left and right boundary $\widehat{r}$ and $\widehat{l}$ respectively is given by 
$$
W = W(\widehat{r}, \widehat{l}) := \{ (x, u) \in  \R^2 : T < u < s, \widehat{r}(u) < x < \widehat{l}(u)\}
$$ 
and termed as the {\it wedge} with left and right boundary $\widehat{r}$ and $\widehat{l}$.
We say that a path $\pi \in \Pi$ enters an open set $A \subset \R^2$ from {\it outside} if
there exists $\sigma_\pi < t_1 < t_2$ such that $\pi(t_1) \notin \bar{A}$ and 
$\pi(t_2) \in A$. The following theorem gives the wedge characterization of the Brownian net.
We state it for the choice $d=1$ only. The general version for $d > 0$ holds similarly.
\begin{theorem}[Theorem 1.10 of \cite{SS08}]
Let $(\WW^l, \WW^r, \widehat{\WW}^l, \widehat{\WW}^r)$ be the 
standard left-right Brownian web and its dual. 
Then, almost surely  
$$
\NN = \{ \pi \in \Pi : \pi \text{ does not enter any wedge of }(\widehat{\WW}^l, \widehat{\WW}^r)
\text{ from outside}\}.
$$
is the Brownian net associated with $(\WW^l, \WW^r)$.
\end{theorem}
 
\subsection{Convergence conditions to the Brownian net}

For a sequence of $( \mathcal{H}, \mathcal{B}_{\mathcal{H}})$-valued random variables 
$\{ \Xi_n : n \geq 1\}$, we consider the following conditions:
\begin{itemize}
\item[(C)]{\bf Non-crossing condition:}
There exist subsets of non-crossing paths $\Xi^l_n, \Xi^r_n \subset \Xi_n$, 
such that no path $\pi \in
\Xi_n$ crosses any path $l \in \Xi^l_n$ from right to left, i.e., $\pi(s) > l(s)$ 
and $\pi(t) < l(t)$ for some $s < t$, and no path $\pi \in  \Xi_n$
 crosses any path $r \in \Xi^r_n$ from left to right. 

\item[(I)]{\bf Finite dimensional distributional convergence:}  For any $\z \in \R^2$ 
there exist $l^{\z}_n \in \Xi^l_n$ and $r^{\z}_n \in \Xi^r_n$, such that for
any deterministic $\z_1,\cdots,\z_k , \z^\prime_1,\cdots,
\z^\prime_{k^\prime} \in \R^2$, the collection $\bigl( l^{\z_1}_n ,\cdots,
l^{\z_k}_n , r^{\z^\prime_1}_n , \cdots, r^{\z^\prime_{k^\prime}}_n \bigr )$ 
converges in distribution to a 
collection of left-right Brownian motions starting
at $(\z_1, \cdots , \z_k, \z^\prime_1, \cdots , \z^\prime_{k^\prime})$ 
with drift $d$, described as in (\ref{eq:LR_SDE}).

\item[(H)]{\bf Hopping condition:}  Almost surely, all the paths in $\Xi_n$ can be obtained by hopping of paths in $\Xi^l_n \cup  \Xi^r_n $ at
intersection times, defined as in  (\ref{def:Hop}).

\item[(U)] {\bf Wedge condition:} Before stating this condition we remark that 
if $\{ \Xi_n : n \geq 1\}$ satisfies conditions (C) and (I), then these two conditions 
together would imply that the sequence $\{ \Xi_n : n \geq 1\}$  is tight.  
There exist $\widehat{\Xi}^l_n, \widehat{\Xi}^r_n \in 
\widehat{\mathcal{H}}$, such that the collection of starting points 
of paths in $\widehat{\Xi}^l_n$ (as well as in $\widehat{\Xi}^r_n$) become dense in $\R^2$ as $n \to \infty$, and paths in $\widehat{\Xi}^l_n$ and $\Xi^l_n$ 
(resp. paths in $\widehat{\Xi}^r_n$ and $\Xi^r_n$) do
not cross a.s. Tightness of the sequence $\{ ( \Xi^l_n, \Xi^r_n ) : n \geq 1 \}$ implies tightness of the dual sequence $\{ ( \widehat{\Xi}^l_n, \widehat{\Xi}^r_n ) : n \geq 1 \}$ too. For any sub-sequential weak limit $(\mathcal{Z}_d , \mathcal{Z}^l_d, \mathcal{Z}^r_d, 
\widehat{\mathcal{Z}}^l_d, \widehat{\mathcal{Z}}^r_d)$ 
of the sequence $\{(\Xi_n, \Xi^l_n, \Xi^r_n, \widehat{\Xi}^l_n,\widehat{\Xi}^r_n) : n \geq 1\}$, paths in $\Xi_d$ do not enter any wedge of 
$\bigl(  (\widehat{\mathcal{Z}}^l_d)^{\Q^2}, 
(\widehat{\mathcal{Z}}^r_d)^{\Q^2} \bigr)$ (defined as in (\ref{def:PathsStartingPoints}))
from outside. 
\end{itemize}

We are now ready to state the first convergence condition to the Brownian net from \cite{SSS19}. 
\begin{theorem}[Theorem 6.11 of \cite{SSS19}]
\label{thm:ConvNet_1}
If $\{ \Xi_n : n \geq 1\}$ satisfies conditions (C), (I), (H), (U) as mentioned above 
then $\Xi_n$ converges in distribution to the Brownian net $\NN_d$ with drift parameter $d$.
\end{theorem}

The version below of Theorem \ref{thm:ConvNet_1}, which is adapted from Remark 6.12 \cite{SSS19} is more applicable to us.
\begin{itemize}
\item[(V)] For a sequence $\{ \Xi_n : n \geq 1\}$ satisfying conditions 
(C) and (I), 
there exists a dual sequence $\{ (\widehat{\Xi}^l_n, \widehat{\Xi}^r_n) : 
n \geq 1\}$ such that paths in $\widehat{\Xi}^l_n$ and $\Xi^l_n$ 
(resp. paths in $\widehat{\Xi}^r_n$ and $\Xi^r_n$) do
not cross a.s., the sequence $\{ ( \Xi^l_n, \widehat{\Xi}^l_n, 
\Xi^r_n, \widehat{\Xi}^r_n ) : n \geq 1 \}$ is jointly tight and the 
collection of starting points of paths in $\widehat{\Xi}^l$ 
(as well as in $\widehat{\Xi}^r$) become dense in $\R^2$. 
We consider the sequence $\{ \widehat{\Xi}_n  : n \geq 1\}$
where $\widehat{\Xi}_n := \text{Hop}(\widehat{\Xi}^l_n, \widehat{\Xi}^r_n)$ for $n \geq 1$.
It follows that the sequence  
$$
\{ (\Xi_n, \widehat{\Xi}_n, \Xi^l_n, \Xi^r_n,  \widehat{\Xi}^l_n, 
\widehat{\Xi}^r_n) : n \geq 1\}
$$
is jointly tight. For any sub-sequential weak limit 
$( \mathcal{Z}^l_d, \mathcal{Z}^r_d, \widehat{\mathcal{Z}}^l_d, \widehat{\mathcal{Z}}^r_d)$ 
of $(\Xi^l_n, \Xi^r_n, \widehat{\Xi}^l_n, \widehat{\Xi}^r_n)$ 
we have joint convergence of first intersection times of scaled forward paths and  
dual paths as well.
\end{itemize}

\begin{theorem}
\label{thm:ConvNet_2}
If $\{ \Xi_n : n \geq 1\}$ satisfies conditions (C), (I), (H), (V) as mentioned above 
then $(\Xi_n, \widehat{\Xi}_n)$ converges in distribution to 
the Brownian net and its dual $(\NN_d, \widehat{\NN}_d)$ with drift parameter $d $
where $\widehat{\NN}_d$ is the Brownian net associated with the dual left-right 
Brownian web.
\end{theorem}


Theorem \ref{thm:ConvNet_2} would be our main tool to prove Theorem \ref{thm:Net_Conv}.
To prove Theorem \ref{thm:ConvNet_2} for our model,  
we construct a dual graph for the Howard network in the next section and 
we discuss its extension in presence of branching.

\section{A dual process for the Howard's model and the effect of branching}
\label{sec:Howard_Dual}

From the discussion on convergence conditions to the Brownian net particularly Theorem 
\ref{thm:ConvNet_2}, it is clear that we need to construct a `good' dual process. 
Following \cite{RSS16B}, we describe a dual process for the `original' Howard network, i.e.,  without branching. 

For $(x,t) \in \Z/2 \times \Z$ we define 
\begin{align*}
(r(x,t), l(x,t)) := 
\begin{cases}
((x + J^+_{(x,t)}, t), (x - J^-_{(x,t)}, t)) & \text{ if } x \in \Z\\
((x + J^+_{(x-1/2,t)}, t), (x - J^-_{(x+1/2,t)}, t)) & \text{ if } x \notin \Z.
\end{cases}
\end{align*}
We observe that $r(x,t)$ and  $l(x,t)$ respectively 
denote the right open and the left open neighbours of $(x,t)$. 
In order to construct the dual process, we need to describe the dual vertex set first. 
The idea is that, for each $m \in \Z$, on the line $y = m$,  we will place a 
dual vertex midway between any two consecutive open vertices. 
Therefore, for $(x,t) \in V$, the notation $\widehat{r}(x,t) := (x + J^+_{(x,t)}/2 , t)$ denotes 
the right dual neighbour. Similarly,  $\widehat{l}(x,t) := (x - J^-_{(x,t)}/2 , t)$ 
denotes the left dual neighbour. The dual vertex set is given by 
$$
\widehat{V} := \{ \widehat{r}(x,t), \widehat{l}(x,t) : (x,t) \in V\} \subset \Z/2 \times \Z.
$$
From each dual vertex $(y,s) \in \widehat{V}$, we define exactly one outgoing dual edge 
to some other dual vertex $(w,s-1) \in \widehat{V}$ such that it does not cross any forward edge in $E$. We do that as follows:

For $(y,s) \in \widehat{V}$, we define   $a^l(y,s), a^r(y,s) \in V$ with 
$a^l(y,s)(2) = a^r(y,s)(2) = s-1$ such that
\begin{align*}
a^l(y,s)(1) & := \sup\{ z : (z,s-1) \in V \text{ and }h(z,s-1)(1) < y\} \text{ and },\\
a^r(y,s)(1) & := \inf\{ z : (z,s-1) \in V \text{ and }h(z,s-1)(1) > y\}.
\end{align*}
We now define the dual step $\widehat{h}(y,s)$ as 
\begin{align*}
\widehat{h}(y,s) := 
\begin{cases}
(a^l(y,s) + a^r(y,s))/2 & \text{ if }y \notin Z \text{ or if } (y,s-1) \notin V\\
((y,s-1) + a^l(y,s))/2  & \text{ if }(y,s-1) \in V \text{ with }h(y,s-1)(1) > y \\
((y,s-1) + a^r(y,s))/2  & \text{ if }(y,s-1) \in V \text{ with }h(y,s-1)(1) < y .
\end{cases}
\end{align*}
We observe that 
\begin{align*}
(a^l(y,s)(1) ,a^r(y,s)(1)) := 
\begin{cases}
(y - J^{-}_{(y,s-1)}, y + J^{+}_{(y,s-1)}) & \text{ if }y \in Z \\
(y - J^{-}_{(y+1/2,s-1)}, y + J^{+}_{(y-1/2,s-1)})  & \text{ if }y \notin Z .
\end{cases}
\end{align*}
Finally, the dual network is given by $\widehat{G} := ( \widehat{V}, \widehat{E})$, where 
$\widehat{E} := \{\langle (y,s), \widehat{h}(y,s)\rangle : (y,s) \in \widehat{V}\}$  denotes the set of 
dual edges.

As in the case of forward edges, for $k \geq 1$, the $k$-th dual 
step starting from $(y,s)\in \widehat{V}$  is denoted by $\widehat{h}^k(y,s)$. 
Joining successive dual steps linearly we get the dual path $\widehat{\pi}^{(y,s)} \in 
\widehat{\Pi}$ starting at $(y,s)$. The collection of all dual paths, 
denoted by $\widehat{\Xi} := \{ \widehat{\pi}^{(y,s)} : (y,s) \in \widehat{V}\}$, 
after taking closure gives a $(\widehat{{\mathcal{ H }}}, d_{\widehat{{\mathcal{ H }}}})$ valued random variable.
The $n$-th order diffusively scaled version is defined as $S_n(\widehat{\Xi})$.
Recall that $\Xi$ denotes the collection of paths obtained from the Howard network (see 
(\ref{def:HowardPath})). 
Theorem 2.9 of \cite{RSS16B} gives joint convergence for the 
scaled Howard network and it's dual to the double 
Brownian web. An immediate corollary of the above result is that 
for any fixed $(y,s) \in \R^2$ there exists a sequence $\{ (y^n, s^n) \in \widehat{V}
: n \geq 1\}$ almost surely such that 
$$
S_n(y^n, s^n) \to (y,s) \text{ as }n \to \infty.
$$

\begin{theorem}[Theorem 2.9 of \cite{RSS16B}]
\label{thm:JtConvBweb}
As $n \to \infty$, $\overline{S_n(\Xi , \widehat{\Xi})}$
 converges in distribution to the standard double 
Brownian web $(\WW, \widehat{\WW})$.
\end{theorem}    
\begin{remark}
In Theorem 2.9 of \cite{RSS16B}, the authors actually proved convergence to the double Brownian web with diffusion coefficient $\sigma(p) = \frac{(1-p)(2 - 2p +p^2)}{p^2(2-p)^2}$, 
whereas in this work, we have incorporated the diffusion coefficient $\sigma$ in the scaling itself (see (\ref{def:DiffusiveScaling})) and therefore, convergence to the standard
double Brownian web follows.   
\end{remark}
Next we extend the construction of dual process in case of branching.
The branching for the forward (original) Howard model induces branching 
for the dual model as well. For $(y,s) \in \widehat{V} $,  
the right dual step, which would not cross 
any of the rightmost (forward) edges, (corresponding to the branching probability $\beta_n$)
is denoted by $\widehat{h}_{n,(r)}(y,s)$ and is defined as
\begin{align*}
\widehat{h}_{n,(r)}(y,s) := 
\begin{cases}
\widehat{h}(y,s) & \text{ if }(y,s-1) \notin V \\
\widehat{h}(y,s) & \text{ if }(y,s-1) \in V  \text{ and }\text{Br}^n_{(y,s-1)} = 0\\
\frac{((y,s-1) + l(y,s-1))}{2} =  (y - \frac{J^-_{(y,s-1)}}{2},s-1) & 
\text{ if }(y,s-1) \in V  \text{ and }\text{Br}^n_{(y,s-1)} = 1.
\end{cases}
\end{align*} 
Similarly, the leftmost step $h_{n,(l)}(x,t)$ is defined as follows:
\begin{align*}
\widehat{h}_{n,(l)}(y,s) := 
\begin{cases}
\widehat{h}(y,s) & \text{ if }(y,s-1) \notin V \\
\widehat{h}(y,s) & \text{ if }(y,s-1) \in V  \text{ and }\text{Br}^n_{(y,s-1)} = 0\\
\frac{((y,s-1) + r(y,s-1))}{2} = (y + \frac{J^+_{(y,s-1)}}{2},s-1) & 
\text{ if }(y,s-1) \in V  \text{ and }\text{Br}^n_{(y,s-1)} = 1.
\end{cases}
\end{align*}
For $k \geq 1$, the $k$-th right  and left dual steps starting 
from $(y,s) \in \widehat{V}$, denoted by $\widehat{h}^k_{n,(r)}(y,s)$ 
and $\widehat{h}^k_{n,(l)}(y,s)$, 
are defined similarly. This allows us to define the right dual path and 
left dual path starting from $(y,s) \in \widehat{V}$ as $\widehat{\pi}^{(y,s),n}_r$ and  
$\widehat{\pi}^{(y,s),n}_l$ respectively obtained by 
linearly joining the right dual steps and left 
dual steps with branching parameter/probability $\beta_n$. 
The collections, $\widehat{\Xi}^n_{r} := \{ \widehat{\pi}^{(y,s),n}_r : (y,s) \in 
\widehat{V} \}$ and $\widehat{\Xi}^n_{l}  := \{ \widehat{\pi}^{(y,s),n}_l : (y,s) \in \widehat{V} \}$, are  defined 
similarly. 

The collection of all dual outgoing edges in the branching Howard model 
with branching probability $\beta_n$ is defined as 
\begin{align*}
\widehat{E}^n := \{ \langle (y,s), \widehat{h}_{n, (r)}(y,s)\rangle , 
\langle (y,s), \widehat{h}_{ n, (l)}(y,s)\rangle : (y,s) \in \widehat{V}\}.
\end{align*}
We define the corresponding collection of paths as 
\begin{align*}
\widehat{\NN}^n := \Bigl \{ & \widehat{\pi} \in \widehat{\Pi} : 
(\widehat{\pi}(\sigma_{\widehat{\pi}}),
\sigma_{\widehat{\pi}}) = (y,s) \text{ for some }(y,s)\in \widehat{V}
\text{ and for all }k \in (-\infty, s] \cap \Z \text{ we have }    \\
& (\widehat{\pi}(k), k) \in  \{ \widehat{h}^{m_1}_{n, (l)}(\widehat{h}^{n_1}_{n, (r)}
(\cdots (\widehat{h}^{n_j}_{n, (r)}(y,s)))) : 
m_1, n_1, \cdots , m_j,  n_j \geq 0 \text{ with }\sum_{i=1}^j(m_i + n_i) = s - k\}\\
&   \text{ and for any }s^\prime \in (k-1,k) \text{ linearly join }
(\widehat{\pi}(k-1), k-1), (\widehat{\pi}(k), k) \Bigr \}.
\end{align*}
We actually prove a stronger version of Theorem \ref{thm:Net_Conv}
as stated in Theorem \ref{thm:DoubleNet_Conv}.
\begin{theorem}
\label{thm:DoubleNet_Conv}
We have 
$$
\overline{S_n(\NN^n, \widehat{\NN}^n )} \Rightarrow (\mathcal{Z}_b, \widehat{\mathcal{Z}}_b) \text{ as }
n \to \infty,
$$
where $b \in [0,\infty)$ is as in (\ref{eq:Limit_b}). About the
 limiting random variable, we  have that 
\begin{align*}
(\mathcal{Z}_b, \widehat{\mathcal{Z}}_b) \stackrel{d}{=}
\begin{cases}
(\mathcal{N}_{\bar{b}}, \widehat{\mathcal{N}}_{\bar{b}}) & \text{ if } b > 0\\
(\mathcal{W}, \widehat{\mathcal{W}}) & \text{ if } b = 0~,
\end{cases}
\end{align*}
where ${\bar{b}} = b \frac{(1-p)}{(2-p)^2} > 0$. 
\end{theorem}
In Theorem \ref{thm:DoubleNet_Conv}, we note that  $(\NN_{\bar{b}}, \widehat{\NN}_{\bar{b}})$ 
denotes the double Brownian net with drift parameter ${\bar{b}}$.   
$\widehat{\NN}_{\bar{b}}$ denotes the dual Brownian net associated with 
the dual left-right Brownian web $(\widehat{\WW}_{\bar{b}}^l, \widehat{\WW}_{\bar{b}}^r)$
with drift ${\bar{b}}$. As a first step towards this, 
Santos et. al. \cite{SVZ24} showed that the left-right diffusively scaled Howard web 
$\overline{S_n(\Xi^n_r, \Xi^n_l)}$ converges in distribution 
to the left-right Brownian web $(\WW^l, \WW^r)$ with drift $\frac{(1-p)}{(2-p)^2}$
(see Theorem 3.6 of \cite{SVZ24}). 
The non-crossing nature of each of these networks 
implies joint convergence of the forward network and the 
dual one as stated in Theorem 3.6 of \cite{SVZ24}. 
Further, this implies that the sequence $\{ \overline{S_n(\NN^n)} : n \geq 1\}$ is tight
and any sub-sequential limit contains a copy of the Brownian net.
We highlight that Theorem 3.6 of \cite{SVZ24} does not cover the case $b = 0$,
which has been considered in Theorem \ref{thm:ConvLeftRight_IntersecTime_Jt}.

\begin{theorem}[Theorem 3.6 of \cite{SVZ24}]
\label{thm:ConvLeftRight_Jt}
For $b > 0$, as $n \to \infty$ we have 
$$
\overline{S_n \bigl ( \Xi^n_r,\widehat{\Xi}^n_r, \Xi^n_l,\widehat{\Xi}^n_l \bigr )} 
\Rightarrow (\WW^r_{\bar{b}}, \widehat{\WW}^r_{\bar{b}}, \WW^l_{\bar{b}}, 
\widehat{\WW}^l_{\bar{b}}), 
$$
where $(\WW^r_{\bar{b}}, \widehat{\WW}^r_{\bar{b}}, \WW^l_{\bar{b}}, 
\widehat{\WW}^l_{\bar{b}})$ denotes double left-right 
Brownian web with drift parameter $\bar{b} = \bar{b}(p) = b\frac{(1-p)}{(2-p)^2}$.
\end{theorem}

Note that the above convergence is in the path-space topology, and therefore, 
joint convergence of first intersection times of paths is {\it not} automatic.   
We prove a stronger version of Theorem \ref{thm:ConvLeftRight_Jt} 
in Theorem \ref{thm:ConvLeftRight_IntersecTime_Jt}, which proves joint convergence of first 
intersection times (for a pair of dual rightmost path and dual leftmost path). 
Because of Theorem \ref{thm:ConvNet_2}, proving Theorem \ref{thm:ConvLeftRight_IntersecTime_Jt}  is sufficient to prove Theorem \ref{thm:DoubleNet_Conv}. Before stating Theorem \ref{thm:ConvLeftRight_IntersecTime_Jt} 
we observe that Theorem \ref{thm:ConvLeftRight_Jt} implies that the sequence 
$$
\bigl\{ 
\overline{S_n \bigl(\Xi^n_r,\widehat{\Xi}^n_r, \Xi^n_l,\widehat{\Xi}^n_l, 
\NN^n, \widehat{\NN}^n \bigr)} : n \geq 1 \bigr \}
$$
satisfies the same set up as that of Theorem \ref{thm:ConvNet_2}. In particular, 
the sequence 
$
\{ \overline{S_n (\NN^n, \widehat{\NN}^n)} : n \geq 1 \}
$ is tight. This was observed in \cite{SVZ24} too.
We are now ready to state Theorem \ref{thm:ConvLeftRight_IntersecTime_Jt}.
For all $\z_1, \z_2 \in \R^2$, let $(l^{\z_1}_n, r^{\z_1}_n) \in (\Xi^l_n, \Xi^r_n)$ and 
$(\widehat{l}^{\z_2}_n, \widehat{r}^{\z_2}_n) \in (\widehat{\Xi}^l_n, \widehat{\Xi}^r_n)$
such that we have 
$$
S_n(l^{\z_1}_n, r^{\z_1}_n, \widehat{l}^{\z_2}_n, \widehat{r}^{\z_2}_n) \to (l^{\z_1}, r^{\z_1}, \widehat{l}^{\z_2},\widehat{r}^{\z_2}) \text{ as }n \to \infty ,
$$  
where $(l^{\z}, r^{\z})$ denotes the left-right (forward) 
 Brownian motions starting from $\z$ with drift parameter $\bar{b}$ 
and $(\widehat{l}^{\z^\prime}, \widehat{r}^{\z^\prime})$ denotes the respective dual versions starting from $\z^\prime$.
Let $\widehat{\tau}^{n}_{\x , \y}$ denote the unscaled first intersection time between 
the paths $\widehat{r}^{\x}_{n}$ and $\widehat{l}^{\y}_{n}$.
Let $\tau^{n}_{\x, \y}$ denote the same for the forward 
paths $r^{\x}_{n}$ and $l^{\y}_{n}$.
On the other hand, let $\widehat{\tau}_{ \x, \y}$ 
denote the first intersection time between   
$\widehat{r}^{\x} \in \widehat{\WW}^r_{\bar{b}}$ and  
$\widehat{l}^{\y} \in \widehat{\WW}^l_{\bar{b}}$ and $\tau_{\x, \y}$
 denotes the same for forward paths $r^{\x}  \in \mathcal{W}^r_{\bar{b}}$ and 
 $l^{\y} \in \mathcal{W}^l_{\bar{b}}$. Both the random variables, 
 $\widehat{\tau}_{ \x, \y}$ and $\tau_{\x, \y}$, need not be finite.
 We are now ready to state Theorem \ref{thm:ConvLeftRight_IntersecTime_Jt} 
 which is a stronger version of Theorem \ref{thm:ConvLeftRight_Jt}.

\begin{theorem}
\label{thm:ConvLeftRight_IntersecTime_Jt}
 As $n \to \infty$, $
\overline{\Bigl( S_n \Bigl( (\Xi^n_r,\widehat{\Xi}^n_r), (\Xi^n_l, \widehat{\Xi}^n_l),  
(\NN^n , \widehat{\NN}^n ),
\bigl ( {\tau}^n_{\x,\y }, \widehat{\tau}^n_{\x,\y } \bigr )_{\x, \y \in \Q^2}  \Bigr )
 \Bigr )}                                                                                                                                                                                                                                                                                                                                                                                                                                                                                                                                                                                                                                                                                                                                                 $
converges in distribution to $\bigl((\mathcal{W}^r_{b}, \widehat{\mathcal{W}}^r_{b}), 
( \mathcal{W}^l_{b}, \widehat{\mathcal{W}}^l_{b}), (\mathcal{Z}_b, \widehat{\mathcal{Z}}_b), 
(\tau_{\x,\y}, \widehat{\tau}_{\x, \y})_{\x, \y \in \Q^2 } \bigr )$.
\end{theorem}
Because of Theorem \ref{thm:ConvNet_2}, Theorem \ref{thm:ConvLeftRight_IntersecTime_Jt} immediately gives us the following Corollary which completes the proof Theorem \ref{thm:DoubleNet_Conv}.
\begin{corollary}
\label{cor:ConvLeftRight_IntersecTime_Jt}
Consider the $(\mathcal{H}\times \widehat{\mathcal{H}}, \mathcal{B}_{\mathcal{H}\times \widehat{\mathcal{H}}})$ valued r.v. 
$(\mathcal{Z}_{b}, \widehat{\mathcal{Z}}_{b})$ as in Theorem \ref{thm:DoubleNet_Conv}. 
Then we have that 
\begin{align*}
(\mathcal{Z}_{b}, \widehat{\mathcal{Z}}_{b}) = 
\begin{cases}
(\mathcal{N}_{\bar{b}}, \widehat{\mathcal{N}}_{\bar{b}}) & \text{ if }b > 0\\
(\WW, \widehat{\WW}) & \text{ if }b = 0.
\end{cases}
\end{align*}
\end{corollary}
Theorem \ref{thm:ConvLeftRight_IntersecTime_Jt} will be proved in Section \ref{sec:ConvDoubleNet}.


\section{Evolution of dual Howard paths}
\label{sec:DualPath_Evolution}

\subsection{Evolution of a single dual Howard path}
\label{subsec:DualPath_Evolution_Marginal}

 
In this section we first discuss the construction of the dual path in more detail.
We will discuss the construction of three types of dual paths or dual steps. 
Recall that for $(x,t) \in \widehat{V} \subseteq  \Z/2\times \Z$ 
and for $k \geq 1$ the notation $\widehat{h}^{k}(x,t)$
denotes the $k$-th step for the general dual path starting from $(x,t)$. On the other hand, 
$\widehat{h}^{k}_{n, (r)}(x,t)$ and $\widehat{h}^{k}_{n, (l)}(x,t)$ respectively denote 
the $k$-th right and left dual step corresponding to the branching probability 
$\beta_n$. If $n$ is fixed then for the $n$-th scaled  model we denote $\widehat{h}^{k}_{n, (r)}(x,t), 
\widehat{h}^{k}_{n, (l)}(x,t)$ simply as $\widehat{h}^{k}_{(r)}(x,t), 
\widehat{h}^{k}_{(l)}(x,t)$. Further, for consistency 
we use the notation $\widehat{h}^{k}_{(o)}(x,t)$ to denote the $k$-th `ordinary'
dual step $\widehat{h}^{k}(x,t)$.
For $\alpha \in \{ l, r, o \}$, the `$\alpha$'-type dual step is denoted as  
$\widehat{h}^{k}_{(\alpha)}(x,t)$.  

For $(x,t) \in \widehat{V}$ and for $\alpha \in \{ l, r, o \}$,
 in order to describe the $\alpha$-type dual step $\widehat{h}_{(\alpha)}(x,t)$ we consider the following 
situations:
\begin{itemize}
\item[(i)] {\bf Case 1:} If $x \notin \Z$ then for any $\alpha \in \{ l, r, o \}$ we have
$$
\widehat{h}_{(\alpha)}(x,t) =  \bigl (l(x,t-1) + r(x,t-1) \bigr )/2 =
(x,t-1) + ((J^+_{(x-1/2,t-1)} - J^-_{(x+1/2,t-1)})/2, 0).
$$
We observe that for the step $\widehat{h}_{(\alpha)}(x,t)$ we have 
$\widehat{h}_{(\alpha)}(x,t)(1) = x + Y^{(1)}$ where the increment
random variable $\widehat{Y}^{(1)}$ does not depend on $\alpha \in \{ l, r, o\}$
and it is given by  
\begin{align}
\label{eq:DualPathIncrement_1}
Y^{(1)} \stackrel{d}{=}  \bigl (J^+_{(x-1/2,t-1)} - J^-_{(x + 1/2,t-1)}  \bigr )/2 
\stackrel{d}{=} ( J_1 - J_2 )/2 ,
\end{align}
where $J_1, J_2 $ are i.i.d. copies of geometric random variable $J$ whose distribution is described in (\ref{def:Geometric}).
Clearly, we have $\E(\widehat{Y}^{(1)}) = 0$.

\item[(ii)] {\bf Case 2:}  The case $x \in \Z$ is considered in several sub-cases.
We first consider either $x \in \Z $ but $(x,t-1) \notin V$ or 
$(x,t-1) \in V$ with $\text{Br}^n_{(x,t-1)} = 0$. In either of these two cases, we 
must have 
$$
\widehat{h}_{(r)}(x,t-1) = \widehat{h}_{(l)}(x,t-1) = \widehat{h}_{(o)}(x,t-1),
$$
and the dual step $h_{(o)}(x,t-1)$ is given by: 
\begin{align*}
\widehat{h}_{(o)}(x,t) = 
\begin{cases}
\bigl ( r(x,t-1) + l(x,t-1)  \bigr )/2 & \text{ if }x \in \Z \text{ but } (x,t-1) \notin V \\
\bigl ( (x,t-1) + l(x,t-1)  \bigr )/2 & \text{ if }(x,t-1) \in V, R_{(x,t-1)} = + 1 \\
\bigl ( (x,t-1) + r(x,t-1)  \bigr )/2 & \text{ if }(x,t-1) \in V, R_{(x,t-1)} = -1 .
\end{cases}
\end{align*}
We denote the increment random variable simply as $Y^{(2)}$ which 
does not depend on $\alpha \in \{ l, r, 0\}$ and it is given by 
\begin{align}
\label{def:DualIncrement_2}
Y^{(2)} := 
\begin{cases}
 (J^+_{(x,t-1)} - J^-_{(x,t-1)})/2 & \text{ if } x \in \Z \text{ but } (x,t-1) \notin V \\
 -J^-_{(x,t-1)}/2  & \text{ if } (x,t-1) \in V, R_{(x,t-1)} = + 1 \\
J^+_{(x,t-1)}/2  & \text{ if } (x,t-1) \in V, R_{(x,t-1)} = - 1 .
\end{cases}
\end{align}
It is not difficult to see that in this case also, 
we have $\E( Y^{(2)}) = 0$ and for any $\ell \in \Z/2$ 
$$
\P( Y^{(2)} = \ell ) = 
(1-p)\P( J_1 - J_2 = 2\ell ) + p\bigl ( \mathbf{1}_{\{ \ell > 0\}} \P( J_1 = 2\ell )
+ \mathbf{1}_{\{\ell < 0\}} \P( J_2 = -2\ell ) \bigr ),
$$ 
where $J_1, J_2$ are i.i.d. copies of $J$ (defined as in (\ref{def:Geometric}))
and $\mathbf{1}_{A}$ denotes the indicator random variable 
corresponding to the event $A$.

We observe that on the event 
$\{\widehat{h}_{(\alpha)}(y,s) = (x,t) \in \Z^2, (x,t-1)\in V\}$
we also have $(x,t-1) \in V_{\text{sp}}$ and a branching event for the forward path at 
$(x,t-1)$, i.e., the occurrence of the event $\{\text{Br}^n_{(x,t-1)} = 1\}$ triggers a branching event 
for the dual path at $(x,t)$ as well. Therefore, if $(x,t-1) \in V$ 
with $\text{Br}^n_{(x,t-1)} = 1$, then the increment random variables 
$$
Y^{(2)}_{(l)} := \widehat{h}_{(l)}(x,t)(1) - \widehat{h}_{(l-1)}(x,t)(1) \text{ and } 
Y^{(2)}_{(r)} := \widehat{h}_{(r)}(x,t) - \widehat{h}_{(r)}(x,t)(1)
$$ 
are given by 
\begin{align*}
(Y^{(2)}_{(l)}, Y^{(2)}_{(r)}) = (J^+_{(x,t-1)}/2, J^-_{(x,t-1)}/2) = 
\bigl ( Y^{(2)} + D^{(l)}_{(x,t)} , Y^{(2)} + D^{(r)}_{(x,t)} \bigr ) ,
\end{align*}
where the `drift' adjustment random variables 
$D^{(l)}_{(x,t)}$ and $D^{(r)}_{(x,t)}$ are respectively given by 
\begin{align*}
D^{(l)}_{(x,t)} :=
\begin{cases}
0 & \text{ if }R_{(x,t-1)} = - 1,  \\
(J^+_{(x,t-1)} + J^-_{(x,t-1)})/2 & \text{ if }R_{(x,t-1)} = + 1 \text{ and }
\end{cases}
\end{align*}
 \begin{align*}
D^{(r)}_{(x,t)} :=
\begin{cases}
0 & \text{ if }R_{(x,t-1)} = + 1 \\
-(J^-_{(x,t-1)} + J^+_{(x,t-1)})/2 & \text{ if }R_{(x,t-1)} = - 1 .
\end{cases}
\end{align*}
Note that here we have assumed that $(x,t-1) \in V$ with $\text{Br}^n_{(x,t-1)} = 1$.
More generally, for any $(x,t) \in \widehat{V}$ with $x \in \Z$, 
the definition of the drift random variables can be generalised as  
\begin{align}
\label{def:DriftIncrement}
( D^{(l)}_{(x,t)}, D^{(r)}_{(x,t)}) :=
\begin{cases}
(0,0) & \text{ if either } (x,t-1) \notin V  
\text{ or }\text{Br}^n_{(x,t-1)} = 0,  \\
\bigl ( (J^+_{(x,t-1)} + J^-_{(x,t-1)})/2, 0 \bigr ) & \text{ if }(x,t-1) \in V, 
\text{Br}^n_{(x,t-1)} = 1, R_{(x,t-1)} = + 1 \\
\bigl ( 0, -(J^-_{(x,t-1)} + J^+_{(x,t-1)})/2 \bigr ) & \text{ if }(x,t-1) \in V, 
\text{Br}^n_{(x,t-1)} = 1, R_{(x,t-1)} = - 1 .
\end{cases}
\end{align}
We highlight that for $\alpha \in \{l,r\}$, the increment random variable 
$Y^{(2)}_{(\alpha)}$ and the drift random variable $D^{(\alpha)}_{(x,t)}$ is defined {\it only when} $x \in \Z$.
Further, from the above description it is not difficult to observe that for $(x,t) \in \widehat{V}$ with $x \in \Z$ we have 
\begin{align}
\label{eq:DriftExpectation}
\E( D^{(l)}_{(x,t)}) = - \E( D^{(r)}_{(x,t)}) 
& = \P((x,t-1) \in V, \text{Br}^n_{(x,t-1)} = 1,R_{(x,t-1)} = +1)\times 
\E((J^+_{(x,t-1)} + J^-_{(x,t-1)})/2) \nonumber \\
& = \P((x,t-1) \in V, \text{Br}^n_{(x,t-1)} = 1,R_{(x,t-1)} = +1)\times \E(J) \nonumber \\
& = \beta_n/2~.
\end{align}

\end{itemize}
 
The construction discussed before readily gives us the following proposition. 
The proof is straightforward and we leave it to the reader.
\begin{proposition}
\label{prop:GenDualPath_Markov}
For $(x,t) \in \widehat{V}$ and for any $\alpha \in \{ l, r, o\}$, 
the process $\{ \widehat{h}^k_{(\alpha)}(x,t)(1) : k \geq 0\}$ is a time homogeneous $\Z/2$ 
valued Markov chain which is given by 
\begin{align*}
M^{k+1}_{(\alpha)} & \stackrel{d}{=} 
\begin{cases} M^{k}_{(\alpha)} + 
 \mathbf{1}_{\{ M^{k}_{(\alpha)} \notin \Z\}} Y^{(1)}_{k+1} + 
\mathbf{1}_{\{ M^{k}_{(\alpha)} \in \Z\}} Y^{(2)}_{k+1} & \text{ if }\alpha = o \\
M^{k}_{(\alpha)} + 
 \mathbf{1}_{\{ M^{k}_{(\alpha)} \notin \Z\}} Y^{(1)}_{k+1} + 
\mathbf{1}_{\{ M^{k}_{(\alpha)} \in \Z\}} (Y^{(2)}_{k+1} +  
D_{(\alpha),k+1})  & \text{ if }\alpha =  l, r 
\end{cases} 
\end{align*} 
where $\{ Y^{(1)}_{k}, Y^{(2)}_{k} : k \geq 1\}$ are i.i.d.  
copies of $Y^{(1)}, Y^{(2)}$ respectively defined as in (\ref{eq:DualPathIncrement_1}), (\ref{def:DualIncrement_2}).
\end{proposition}

From the construction of the (ordinary) dual path, we also have the following proposition, 
which was proved earlier in \cite{RSS16B}.
  
\begin{proposition}[Proposition 2.3 of \cite{RSS16B}]
\label{prop:DualPathMartingale}
For any $(y,s) \in \widehat{V}$, the process $\{ M^{k}_{(o)}(y,s) : k \geq 0  \}$ 
is $L^2$-martingale with respect to the filtration 
\begin{align}
\label{def:Filtration}
\{ \mathcal{F}_k := \sigma \bigl( B_{(x,t)}, R_{(x,t)} : (x,t) \in \Z^2, 
s - k \leq t \leq s \bigr ) : k \geq 0\}.
\end{align}
\end{proposition}
We define the process 
\begin{align}
\label{def:Y_Markov}
\{ Z^{(\alpha)}_k := \mathbf{1}_{\{M^{k}_{(\alpha)}(y,s) \in \Z\}} : j \geq 0 \}.
\end{align}
The following proposition describes the behaviour of successive 
hitting times of points in $\Z$ by $\alpha$-type dual process Markov chain 
$\{ M^{k+1}_{(\alpha)} : k \geq 0 \}$.
  
\begin{proposition}
\label{prop:IntegerNonintegerMarkov}
For $\alpha \in \{ l, r, o\}$, the process $\{ Z^{(\alpha)}_j  : j \geq 0 \}$ is a time homogeneous $\{0,1\}$-valued Markov chain with the transition probability matrix given by 
\[\begin{pmatrix}
 \frac{(2-2p+p^2)}{(2-p)^2} & \frac{2(1-p)}{(2-p)^2} \\
 \frac{(2-2p+p^2)}{(2-p)^2} & \frac{2(1-p)}{(2-p)^2} 
\end{pmatrix} .\]
\end{proposition}
\begin{proof}
Fix $(y,s) \in \widehat{V}$ and take $j \geq 1$. 
Suppose $\widehat{h}^{j}_{(\alpha)}(y,s)(1) = u + 1/2$ for some $u \in \Z$, we recall 
that the $\alpha$-type dual step does not depend on $\alpha \in \{ l, r, o\}$. In 
order to have $\widehat{h}^{j+1}_{(\alpha)}(y,s)(1) \notin \Z$  we require that 
$J^+_{(u,s-j-1)} - J^-_{(u+1,s-j-1)} \in 2\Z$.
For any $j \geq 1$ using the fact that $J^+_{(u,s -j-1)}$ and $J^-_{(u+1,s-j-1)}$
are i.i.d. copies of geometric random variable $J$, we have 
\begin{align*}
 & \P(Y_{j+1} = 0| Y_{j} = 0) \\
 & = \P(J^+_{(u, s-j-1)} - J^-_{(u+1,s-j-1)} \in 2\Z) \\
& = \P(J^+_{(u,s-j-1)}, J^-_{(u+1,s-j-1)} \in 2\Z ) + 
\P(J^+_{(u,s-j-1)},J^-_{(u+1,s-j-1)} \notin 2\Z)\\
& = (\P( J \in 2\Z))^2 + (\P( J \notin 2\Z))^2 \\
& = (2-2p+p^2)/(2-p)^2 .
\end{align*}

Suppose $\widehat{h}^{j}_{(\alpha)}(y,s)(1) = u $ for some $u \in \Z$. Then 
in order to have $\widehat{h}^{j+1}_{(\alpha)}(y,s)(1) \in \Z$  we need to consider three cases:
\begin{itemize}
 \item[(i)] if $(u,s-j-1) \in V$ and $h(u,s-j-1)(1) > u$ then we require 
 $J^-_{(u,s-j-1)} \in 2\Z$, 
 \item[(ii)] if $(u,s-j-1) \in V$ and $h(u,s-j-1)(1) < u$ then we require 
 $J^+_{(u,s-j-1)} \in 2\Z$ and
 \item[(iii)] if $(u,s-j-1) \notin V$ then we require 
 $J^+_{(u,s-j-1)} - J^-_{(u,s-j-1)} \in 2\Z$.
\end{itemize}
Recall that $J^+_{(u,s-j-1)}$ and $J^-_{(u,s-j-1)}$ are i.i.d. copies 
of the geometric random variable $J$ and we have that
\begin{align*}
& \P(Y_{j+1} = 1| Y_{j} = 1) \\
& = p/2(\P(J^+_{(u,s-j-1)} \in 2\Z) + \P(J^-_{(u,s-j-1)} \in 2\Z)) + \\
& \qquad \qquad (1-p)( \P(J^+_{(u,s-j-1)} , J^-_{(u,s-j-1)} \in 2\Z) + 
\P(J^+_{(u,s-j-1)} , J^-_{(u,s-j-1)} \notin 2\Z)) \\
& = p\P(J \in 2\Z) + (1-p)((\P( J \in 2\Z))^2 + (\P( J \notin 2\Z))^2) \\
& = 2(1-p)/(2-p)^2 .
\end{align*}
This completes the computation for the transition probability matrix.  
%
%
\end{proof}
For $(y,s) \in \widehat{V}$, set $\beta^{(\alpha)}_0  = \beta^{(\alpha)}_0(y,s) = 0$ 
and for $\ell \geq 1$ let $\beta^{(\alpha)}_\ell$ denote the $\ell$-th hitting time of $\Z$ by an $\alpha$-type dual path starting from $(y,s)$,  defined as,
\begin{align}
\label{def:Lattice_RenewalSteps}
\beta^{(\alpha)}_\ell = \beta^{(\alpha)}_\ell(y,s) := \inf\{ j > 
\beta^{(\alpha)}_{\ell - 1} : \widehat{h}^{j}_{(\alpha)}(y,s)(1) \in \Z \} = \inf\{ j > \beta^{(\alpha)}_{\ell - 1} : Z^{(\alpha)}_j = 1\}. 
\end{align}
The next proposition shows that these $\beta$ steps are actually 
renewal steps for respective $\alpha$-type dual path processes.
We define
 \begin{align}
 \label{eq:Renewal_Increment}
(X^{\ell+ 1}_{(\alpha)}, T^{\ell+ 1}_{(\alpha)}) := 
\bigl ( \widehat{h}^{\beta_{\ell+1}}_{(\alpha)}(y,s)(1) -  
 \widehat{h}^{\beta_\ell}_{(\alpha)}(y,s)(1),
 \widehat{h}^{\beta_{\ell+1}}_{(\alpha)}(y,s)(2) - 
 \widehat{h}^{\beta_\ell}_{(\alpha)}(y,s)(2) \bigr ).
\end{align}    

\begin{proposition}
\label{prop:DualRenewalSteps}
For any $(y,s) \in \widehat{V}$ and for any $\alpha \in \{ l, r, o\}$, 
we have the following:
\begin{itemize}
\item[(i)] For all $\ell, j \geq 1$ we have 
\begin{align}
\label{eq:BetaExpDecay}
\P( \beta^{(\alpha)}_{\ell + 1} - \beta^{(\alpha)}_{\ell} = j) = 
\bigl( \frac{(2 - 2p + p^2)}{(2-p)^2} \bigr )^{j-1}\frac{2(1-p)}{(2-p)^2}~.
\end{align}
This immediately gives that for any $\ell \geq 1$ we have 
\begin{align}
\label{eq:ExpecRenewalTime}
\E (\beta^{(\alpha)}_{\ell + 1} - \beta^{(\alpha)}_{\ell}) = \frac{(2-p)^2}{2(1-p)}.
\end{align}

\item[(ii)] For any $\alpha \in \{ l, r, o\}$, the evolution of the process 
 between successive $\beta^{(\alpha)}$ steps gives rise to i.i.d. renewal blocks 
 and we have 
 \begin{align*}
\bigl \{ (X^{\ell+ 1}_{(\alpha)}, T^{\ell+ 1}_{(\alpha)}) : \ell \geq 1 \bigr \},
\end{align*} 
gives an i.i.d. collection of random vectors taking values in $\Z \times \N$. 
\end{itemize}
\end{proposition}
\begin{proof}
We first consider Item (i). Proposition \ref{prop:GenDualPath_Markov} gives us that 
for any $\alpha \in \{ l, r, o\}$ and for any $(y,s) \in \widehat{V}$ we have 
$$
\P( \widehat{h}_{(\alpha)}(y,s) \in \Z^2) \geq \widehat{p} = 
\frac{2(1-p)}{(2-p)^2}  > 0 .
$$
Therefore, for any $\ell \geq 1$ the r.v. $\beta_{\ell}$ is finite a.s. 
Further, for $(y_1,s_1), (y_2,s_2) \in \widehat{V}$ with $s_1 \neq s_2$, 
we observe that for any $\alpha \in \{ l, r, o\}$, the events, 
$\{\widehat{h}_{(\alpha)}(y_1,s_1) \in \Z^2\}$ and 
$\{\widehat{h}_{(\alpha)}(y_2,s_2) \in \Z^2\}$ are independent. 
Therefore, Proposition \ref{prop:GenDualPath_Markov} gives us that 
the r.v. $(\beta_{\ell + 1} - \beta_\ell)$
follows geometric distribution with success parameter $\frac{2(1-p)}{(2-p)^2}$. This completes the proof of Item (i). 

Item (ii) follows from the Markov property and the 
translation invariance nature of our model.    
\end{proof}


The next Proposition gives us the diffusive scaling limit of an $\alpha$-type dual path.
Recall that for  $(y,s) \in \widehat{V}$, 
the notation $\widehat{\pi}^{(y,s)}_{(\alpha)}$ denotes 
the $\alpha$-type dual (backward) path obtained by linearly
 joining successive $\alpha$-type dual steps. 
\begin{proposition}
\label{prop:DualPathDiffLimit}
For any $(y,s) \in \widehat{V}$,  we have the following 
\begin{itemize}
\item[(i)] $S_n(\widehat{\pi}^{(y,s)}_{(o)})$ converges in distribution to 
the standard backward Brownian motion, i.e., running backward in time, starting from origin. 
\item[(ii)] $S_n(\widehat{\pi}^{(y,s)}_{(l)})$ converges in distribution to 
the backward (dual) Brownian motion with drift $\bar{b}(p)$ (defined as in (\ref{def:Drift_Param})) starting from origin.
\item[(iii)] $S_n(\widehat{\pi}^{(y,s)}_{(r)})$ converges in distribution to 
the backward Brownian motion with drift $-\bar{b}(p)$ starting from origin.
\end{itemize}
\end{proposition}
\begin{proof}
Item (i) of Proposition \ref{prop:DualPathDiffLimit} 
has been proved in Proposition 2.10 of \cite{RSS16B}. 

We prove Item (ii) and the proof for Item (iii) is exactly the same. 
Recall that the drift random variables corresponding to the leftmost dual path 
at the $\ell$-th renewal step is given by $D_{(l),\ell}$.  
Proposition \ref{prop:GenDualPath_Markov} gives us that for any $j \geq 0$ 
we have  
\begin{align*}
\widehat{h}^j_{(l)}(y,s)(1) \stackrel{d}{=} 
\widehat{h}^j_{(o)}(y,s)(1) + 
\sum_{\ell \geq 1 : \beta^{(l)}_\ell  \leq j} D_{(l),\ell} ~.
\end{align*}
Fix any $T > 0$. Because of Item (i), it is enough to prove that almost surely we have
\begin{align}
\label{eq:DriftCalculation_1}
\epsilon_n \bigl ( \sum_{\ell \geq 1 : \beta^{(l)}_\ell  \leq \lfloor T/(\epsilon_n)^2\rfloor} D^{(l)}_{\ell} \bigr )
\to \bar{b}T \text{ as }n \to \infty~.
\end{align}

Let $N^T_{\beta} := \sup\{ \ell \geq 0 : 1 \leq \beta_\ell \leq \lfloor 
T/(\epsilon_n)^2\rfloor \}$ denote the number of $\beta$ renewal steps 
within $\lfloor T/(\epsilon_n)^2\rfloor$ many steps. Clearly, $N^T_{\beta} \to \infty$
almost surely as $n \to \infty$. Therefore, we can write (\ref{eq:DriftCalculation_1}) 
as follows and obtain 
\begin{align}
\label{eq:DriftCalculation_2}
\epsilon_n \bigl ( \sum_{\ell \geq 1 : \beta_\ell \leq \lfloor T/(\epsilon_n)^2\rfloor} 
D^{(r)}_{\beta_\ell}  \bigr ) = \epsilon_n \Bigl ( \bigl ( \sum_{\ell \geq 1 : \beta_\ell \leq \lfloor T/(\epsilon_n)^2\rfloor } D^{(r)}_{\beta_\ell}/ N^T_{\beta} \bigr ) 
\times ( N^T_{\beta} / (\lfloor T/(\epsilon_n)^2 \rfloor )) \times 
\lfloor T/(\epsilon_n)^2 \rfloor \Bigr )
\end{align} 
As $n \to \infty$, because of Proposition \ref{prop:DualRenewalSteps}, for Item (ii) 
we can apply SLLN to ensure that the first term in the r.h.s. of (\ref{eq:DriftCalculation_2}) 
almost surely converges to $\E(D^{(r)}_{\beta_\ell}) = \beta_n/2 $ (see (\ref{eq:DriftExpectation})). By  Item (i) of Proposition \ref{prop:DualRenewalSteps}, the second term in the r.h.s. of (\ref{eq:DriftCalculation_2}), i.e., the ratio 
of the total number of renewals in $\lfloor T/(\epsilon_n)^2) \rfloor$ many steps and $\lfloor T/(\epsilon_n)^2) \rfloor$, converges to 
$(\E (\beta^{(\alpha)}_{\ell + 1} - \beta^{(\alpha)}_{\ell}))^{-1} = 
\frac{2(1-p)}{(2-p)^2}$. Therefore, we obtain    
\begin{align*}
& \lim_{n \to \infty} \epsilon_n \Bigl ( \bigl ( \sum_{\ell \geq 1 : \beta_\ell \leq \lfloor T/(\epsilon_n)^2\rfloor} D^{(r)}_{\beta_\ell}/ N^T_{\beta} \bigr ) 
\times (N^T_{\beta}/ (\lfloor T/(\epsilon_n)^2 \rfloor) 
\times \lfloor T/(\epsilon_n)^2 \rfloor  \Bigr ) \\
& = \lim_{n \to \infty }  (\epsilon_n \times \frac{\beta_n}{2}  
\times \frac{2(1-p)}{(2-p)^2} \times 
\lfloor T/(\epsilon_n)^2 \rfloor) \\
& = b \frac{(1-p)}{(2-p)^2} T = \bar{b}T , 
\end{align*} 
where $\bar{b}$ is as in (\ref{def:Drift_Param}).
This completes the proof.  
\end{proof}

 
For $(y,s) \in \widehat{V}$ instead of considering the actual $\alpha$-type dual path we consider the `modified' dual path obtained by linearly joining the successive renewal steps denoted by $\widehat{\tilde{\pi}}^{(y,s)}_{(\alpha)}$ defined as 
\begin{align}
\label{def:ModDualPath_renewal} 
\widehat{\tilde{\pi}}^{(y,s)}_{(\alpha)}(s - t) := \widehat{h}^{\beta^{(\alpha)}_\ell}_{(\alpha)}(y,s)(1) + \frac{(t - \widehat{h}^{\beta^{(\alpha)}_\ell}_{(\alpha)}(y,s)(2))}{T^{\ell + 1}_{(\alpha)}} X^{\ell + 1}_{(\alpha)}\text{ for }\beta^{(\alpha)}_\ell \leq t <  \beta^{(\alpha)}_{\ell + 1} 
\end{align}
where $\{ T^{\ell}_{(\alpha)} = T^{\ell}_{(\alpha)}(y,s) : \ell \geq 1\}$ denotes the sequence of successive renewal steps and $\{ X^{\ell}_{(\alpha)} = X^{\ell}_{(\alpha)}(y,s) : \ell \geq 1\}$ denotes the corresponding 
sequence of increment random variables.     
Proposition \ref{prop:DualPathDiffLimit} immediately gives us the following corollary. 
\begin{corollary}
\label{cor:DualPathDiffLimit}
For any $(y,s) \in \widehat{V}$,  we have the following 
\begin{itemize}
\item[(i)] $S_n(\widehat{\tilde{\pi}}^{(y,s)}_{(o)})$ converges in distribution to 
the standard backward Brownian motion, i.e., running backward in time, starting from origin. 
\item[(ii)] $S_n(\widehat{\tilde{\pi}}^{(y,s)}_{(l)})$ converges in distribution to 
the backward (dual) Brownian motion with drift $\bar{b}(p)$ (defined as in (\ref{def:Drift_Param})) starting from origin.
\item[(iii)] $S_n(\widehat{\tilde{\pi}}^{(y,s)}_{(r)})$ converges in distribution to 
the backward Brownian motion with drift $-\bar{b}(p)$ starting from origin.
\end{itemize}
\end{corollary}

\subsection{Joint evolution of two dual Howard paths}
\label{subsec:DualPath_Evolution_Joint}

Before discussing the evolution of the joint process, we 
first observe that for any $(y,s) \in \widehat{V}$ and for any $\alpha \in \{l, r, o\}$, 
the construction of the marginal process $\{ \widehat{h}^j_{(\alpha)}(y,s) : j \geq 0 \}$ shows that,
in order to obtain the step $\widehat{h}^j_{(\alpha)}(y,s)$ it is 
{\it enough} to explore random variables $B_{(z,t-1)}, R_{(z,t-1)}, \text{Br}^n_{(z,t-1)}$ 
attached to lattice points inside a region $\widehat{E}^{1}$ which is given by
\begin{align*}
E^{1} = E^{1}(y,s) := \begin{cases}
\{ (z, s  - 1) \in \Z^2 : - J^-_{(y,s-1)} \leq z - y \leq J^+_{(y,s-1)} \} & \text{ if } y \in \Z \\
\{ (z, s  - 1) \in \Z^2 : - J^-_{(y + 1/2,s-1)} \leq z - y \leq J^+_{(y - 1/2,s-1)} \} & \text{ otherwise}.
\end{cases}
\end{align*}
In other words, $\widehat{h}^j_{(\alpha)}(y,s)$ is measurable 
w.r.t. the $\sigma$-field 
$$
\sigma \bigl (\{ (B_{(z, s-1)}, R_{(z, s -1)}, \text{Br}^n_{(z, s-1)}) 
: (z,s-1) \in  E^{1} \} \bigr ).
$$ 
In what follows, we refer the region $E^{1}(y,s)$ as the exploration region 
to obtain the first $\alpha$-type dual step starting from $(y,s)$.  
More generally, for $k \geq 1$ the exploration region for the `$\alpha$' type $k$-th dual step starting from $(y,s)$ is denoted by $E^{k} = E^{k}(y,s)$ 
and is defined similarly.

For $(y_1, 0), (y_2, 0) \in \widehat{V}$ with $y_1 < y_2 $ and for $
\alpha_1, \alpha_2 \in \{ l, r, o\}$  we consider 
the joint evolution of $\alpha_1$ and $\alpha_2$-type dual paths 
$\{ \widehat{h}^k_{(\alpha_1)}(y_1, 0), \widehat{h}^k_{(\alpha_2)}(y_2, 0) : k \geq 0\}$ 
starting from $(y_1, 0)$ and $(y_2, 0)$. Let $\eta = \eta(\alpha_1, \alpha_2)$ denote 
the random step such that the respective exploration regions overlap for the first time:    
\begin{align*}
\eta = \eta(\alpha_1, \alpha_2) := \inf\{ k \geq 1 : E^k(y_1,0) \cap 
E^k(y_2,0) \neq \emptyset \}.
\end{align*}
From the construction of both the dual paths, as described in the previous section, 
the next lemma is straightforward to observe.

\begin{lemma} 
\label{lem:independence}
The processes, $\{ \widehat{h}^j_{(\alpha_1)}(y_1,0) : 0 \leq j < \eta \}$ 
and $ \{ \widehat{h}^j_{(\alpha_2)}(y_2,0): 0 \leq j < \eta \}$ are independent.
\end{lemma} 
\begin{proof}
The proof simply follows from the fact that until the step $\eta$, for both the dual paths 
respective exploration regions remain disjoint. Until the $\eta$-th 
step, the dual paths $\{ \widehat{h}^j_{(\alpha_1)}(y_1,0) : 0 \leq j < \eta \}$ 
and $ \{ \widehat{h}^j_{(\alpha_2)}(y_2,0): 0 \leq j < \eta \}$ respectively use disjoint collections of random variables given by  
\begin{align*}
A_1 := & \cup_{k=1}^\eta \{ (B_{(z,-k)}, R_{(z,-k)}, \text{Br}^n_{(z,-k)}) : 
(z,-k) \in E^k(y_1,0) \} 
\text{ and } \\
A_2 := & \cup_{k=1}^\eta \{ (B_{(z,-k)}, R_{(z,-k)}, 
\text{Br}^n_{(z,-k)}) : (z,-k) \in E^k(y_2,0) \}.
\end{align*}
By definition of the random variable $\eta$ we have 
$A_1 \cap A_2 = \emptyset$ and therefore, 
the proof follows.
\end{proof}
For $\alpha_1, \alpha_2 \in \{ l, r, o\}$ and for 
$k \geq 0$ let $F^{k+1}$ denote the event 
\begin{align}
\label{def:Event_Fk}
F^{k+1} = F^{k+1}(\alpha_1, \alpha_2) := \bigl ( \widehat{h}^{k}_{(\alpha_1)}(y_1,0) 
\neq \widehat{h}^{k}_{(\alpha_2)}(y_2,0), E^{k + 1}(y_1,0) \cap E^{k + 1}(y_2,0) 
\neq \emptyset \bigr ). 
\end{align}
\begin{lemma}
\label{lem:CoalescenceProb}
We have 
\begin{align*}
\P \bigl ( \widehat{h}^{k + 1}_{\epsilon, (\alpha_1)}(y_1,0) = 
\widehat{h}^{k + 1}_{\epsilon, (\alpha_2)}(y_2,0) \mid 
 F^{k+1} \bigr )
 \geq (1-p)^3.
\end{align*}
\end{lemma}
\begin{proof}
It is important to observe that given that the event $F^{k+1}$ has happened, 
the realizations of the random variables 
$(B_{(y,s)},R_{(y,s)}, \text{Br}^n_{(y,s)}) $ for  
$(y,s) \in E^{k + 1}_{(\alpha_1)}(y_1,0) \cup E^{k + 1}_{(\alpha_2)}(y_2,0)$ are {\it not} 
completely known to us. 
Set $ \widehat{h}^{k}_{(\alpha_i)}(y_i,0) = (w_i, -k) \in \widehat{V}$ for $ i \in \{1, 2\}$.
Firstly, we consider the situation that $ w_1 < w_2$ with $w_1, w_2 \notin \Z$. 
We will consider the more general case later.

Given that the event $F^{k+1}$ has occurred,
we observe that there must exist integers $j_1, j_2 \geq 0$ satisfying 
\begin{align}
\label{def:J_1andJ_2}
& \{(m , -k -1) \in \Z^2 :  m - w_1 \in [0, j_1] ,  m - w_2 \in [-j_2 , 0]   \} = \emptyset  \text{ and }  \nonumber \\
& \{(m , -k -1) \in \Z^2 : m - w_1 \in [0, j_1 + 1] , m - w_2 \in [-(j_2+1) , 0]  \} \neq  \emptyset ,  
\end{align}
such that the event $ \bigl ( J^+_{(w_1 - 1/2, -k-1)} > j_1 ,
J^-_{(w_2 + 1/2, -k-1)} > j_2 \bigr ) $ occurs. 
In other words, we have the following event equality: 
\begin{align}
\label{def:EventEk_ExplorationRepresentation}
F^{k+1} = \bigcup_{j_1, j_2 \in \Z^+ \text{ satisfying }(\ref{def:J_1andJ_2})} 
\bigl ( J^+_{(w_1 - 1/2, -k-1)} > j_1, J^-_{(w_2 + 1/2, -k-1)} > j_2 \bigr ).
\end{align} 
Further, we observe that on the event $F^{k+1}$ for $j_1 , j_2 \in \Z^+$ satisfying (\ref{def:J_1andJ_2})
we have that the set 
\begin{align*}
& \{(m , -k -1) \in \Z^2 : m - w_1 \in [0, j_1 + 1] , m - w_2 \in [-(j_2+1) , 0] \}
\end{align*}
must be singleton and suppose $(w, -k-1)$ denotes the common vertex. 

Given that the event $F^{k+1}$ has occurred, $B_{(w, -k-1)}$ is independently distributed as
Bernoulli r.v. with parameter $p$. If $B_{(w, -k-1)} = 0$, 
then the two dual paths must meet at the $(k+1)$-th step. This completes the proof of 
Lemma \ref{lem:CoalescenceProb} under the assumption that $\widehat{h}^k_{(\alpha_1)}(y_1,0), 
\widehat{h}^k_{(\alpha_2)}(y_2,0) \notin \Z$. 

In order to prove Lemma \ref{lem:CoalescenceProb} for more general situation  
conditional to the event $F^{k+1}$, we consider the event 
$$
\{\widehat{h}^k_{(\alpha_1)}(y_1,0) + (0,-1), \widehat{h}^k_{(\alpha_2)}(y_2,0) + (0,-1), (w, -k-1)\}
\cap V = \emptyset ,
$$  
which ensures the intersection of two dual paths at the next step. 
Further, the probability of occurrence of the above event is bounded below by $(1-p)^3$ 
as the cardinality of the set considered in the above mentioned event  
is bounded by three. This completes the proof.
\end{proof}

\section{Division of the total time for joint evolution of two dual paths}
\label{sec:Div_TotalTime}

Because of the translation invariance nature of our model, it is enough to consider 
dual paths starting from $(y_1,0),  (y_2, 0) \in \widehat{V}$ with $y_1 \leq y_2$.
We consider the joint evolution of $\alpha_1$ and $\alpha_2$ type dual paths for
$\alpha_1, \alpha_2 \in \{l, r, o\}$ starting from $(y_1, 0)$ and $(y_2, 0)$ respectively. 
We describe a coupling construction for the same. 
Since, $n$ is fixed at the moment, with a slight abuse of notation we denote the joint process as 
 $$
 \{ \widehat{h}^{j}_{(\alpha_1)}(y_1,0), \widehat{h}^{j}_{(\alpha_2)}(y_2,0) : j \geq 0\}.
 $$ 
 Below, we partition the evolution time for the joint process into three different categories. We highlight that our partition is not exactly based on the 
 current locations of the joint process. It is based on how the joint process 
 is doing the exploration for the next step. It should be mentioned that in 
 order to classify the nature of exploration (for the next step), we {\it 
 don't} need to know the positions at the next step exactly.
   
\begin{definition} 
We say that the $(j+1)$-th step of the joint process, i.e., 
$ ( \widehat{h}^{j+1}_{(\alpha_1)}(y_1,0) ,  \widehat{h}^{j+1}_{(\alpha_2)}(y_2,0) )$ is 
taken 
\begin{itemize}
\item[(i)] in a `common' manner if $\widehat{h}^{j}_{(\alpha_1)}(y_1,0) = 
\widehat{h}^{j}_{(\alpha_2)}(y_2,0) = (w,-j)$ for some $w \in \Z/2$, i.e., the paths were  coalesced at the $j$-th step. 
The nomenclature reflects that the next step of the joint process, 
which may result in a branching event too, is decided based on the `common' collection 
of random variables given by 
\begin{align*}
\{ \Lambda_{(x,- j - 1)} : (x,- j - 1) \in \Z^2 \text{ with }w - J^-_{(w, - j -1)} \leq x \leq w + J^+_{(w, - j -1)}\} & \text{ if } w \in \Z \\
\{ \Lambda_{(x,- j - 1)} : (x,- j - 1) \in \Z^2 \text{ with }w - J^-_{(w + 1/2, - j -1)} 
\leq x \leq w + J^+_{(w - 1/2 , - j -1)}\} & \text{ if } w \notin \Z . 
\end{align*}
\item[(ii)] In `dependent' manner if the event $F^{j+1}$ (defined as in (\ref{def:Event_Fk}))
occurs. Note that by definition, the occurrence of the event $F^{j+1}$ implies 
$\widehat{h}^{j}_{(\alpha_1)}(y_1,0) \neq \widehat{h}^{j}_{(\alpha_2)}(y_2,0)$.
\item[(iii)] In `independent' manner if $\widehat{h}^{j}_{(\alpha_1)}(y_1,0) \neq 
\widehat{h}^{j}_{(\alpha_2)}(y_2,0)$ and the event $F^{j+1}$ does not occur.
\end{itemize} 
\end{definition}
We observe here that entrance to the common state
is one step later than the entrance to the coalesced state, i.e., when 
the two paths got coalesced. The same holds for the exit step as well. As a result 
the total number of steps spent in the common state is same as
the total number of steps spent in the coalesced state.  
Decomposition into these three states gives a partition of the total evolution time
for the joint process of two dual paths.   
It is important to note that in order to know the state of 
the joint process of two dual paths at the $(j+1)$-th step, 
we don't need to know the path positions at the $(j+1)$-th step exactly. 
It is enough to know the respective positions at the $j$-th step 
and whether the exploration regions (for the $(j+1)$-th step) have overlapped or not.

We set $\tau^{n,C}_0 = 0$ and for $j \geq 1$, the $j$-th entrance time to the common state 
is defined as 
\begin{align}
\label{def:Tau^C}
\tau^{n,C}_j & :=  \inf \{ k+1 > \tau^{n,C}_{j - 1 } : (k+1) \text{-th step is taken in common manner but }
k\text{-th step isn't}\}.
\end{align}
For simplicity, we say that the notation $\tau^{n,C}_j$ denotes the $j$-th entrance to the common state. We note that if $\alpha_1 = \alpha_2$, then $\tau^{n,C}_1$ becomes an {\it absorbing state} and we have $\tau^{n,C}_j = \infty$ for all $j > 1$.
Similarly, we set $\tau^{n,D}_0 = \tau^{n,I}_0 = 0$ and for $j \geq 1$, 
the respective $j$-th re-entrance times for the other two states are defined as
\begin{align}
\label{def:Tau^D_Tau^I}
\tau^{n,D}_j & :=   \inf \{ k+1 > \tau^{n,D}_{j - 1 } : 
(k+1) \text{-th step is taken in dependent manner but }k\text{-th step isn't} \} \text{ and }\nonumber\\
\tau^{n,I}_j & :=  \inf \{ k+1 > \tau^{n,C}_{j - 1 } : 
(k+1) \text{-th step taken in independent manner but }k\text{-th step isn't} \}.
\end{align} 
We  consider the filtration 
\begin{align}
\label{def:Filtration_1}
\mathcal{F}_k := \sigma \bigl ( \{ (B_{(x,t)}, R_{(x,t)}, \text{Br}^{n}_{(x,t)} ) 
: (x,t) \in \Z^2, t \in [-k, 0] \} \bigr ),
\end{align}
and observe that for any $j \geq 1$, the r.v.  $\tau^{n,c}_j$ is a stopping time 
w.r.t. the filtration $\{\overline{\mathcal{F}}_k : k \geq 0\}$ 
whereas the r.v.'s  $\tau^{n,D}_j$ and the r.v. $\tau^{n,I}_j$ are not stopping times
 w.r.t. this filtration. Therefore, we extend the filtration as
\begin{align}
\label{def:Filtration_2}
\overline{\mathcal{F}}_k := \sigma\{ \mathbf{1}_{F^1}, \cdots, \mathbf{1}_{F^k}, 
(B_{(x,t)}, R_{(x,t)}, \text{Br}^{n}_{(x,t)} ) 
: (x,t) \in \Z^2, t \in [-k, 0] \},
\end{align}
where for any event $A$, the notation $\mathbf{1}_{A}$ denotes the indicator r.v. 
corresponding to the event $A$. We note that for any $j\geq 1$, all the random variables $\tau^{n,C}_j$, 
$\tau^{n,D}_j$, $\tau^{n,I}_j$ are stopping times w.r.t. the filtration defined as in 
(\ref{def:Filtration_2}). The next remark is straightforward to observe.

\begin{remark}
\label{rem:DepState_CoalState}
We observe that for any $j \geq 1$, the $j$-th entrance to the common state 
must precede with an entrance to the dependent state, i.e., the joint process can enter the common state only from the dependent state.  
\end{remark}
For $j \geq 1$, the number of steps taken to exit after the $j$-th entrance to the independent state, i.e., $\tau^{n,I}_{j}$  is defined as 
\begin{align}
\label{def:Tau^C}
\tau^{n,I}_{j,\text{next}} :=  \inf\{ \tau^{n,D}_{j_1} \wedge \tau^{n,C}_{j_2} : 
\tau^{n,D}_{j_1} , \tau^{n,C}_{j_2} >  
\tau^{n,I}_{j} \text{ for some }j_1, j_2 \geq 1\}.
\end{align}
Similarly, the step $\tau^{n,C}_{j,\text{next}}$ is defined.

\begin{lemma}
\label{lem:CoalescedStateTime}
Fix $k \in \N$ and consider $\alpha_1 = l$ and  $\alpha_2 = r $. Then we have that the r.v.s 
$\{ \tau^{n,C}_{k, \text{next}} - \tau^{n,C}_k : k \geq 1\}$ are {\it independent}
(need not be i.i.d.) with 
\begin{align*}
\E(\tau^{n,C}_{k, \text{next}} - \tau^{n,C}_k ) \geq \frac{1-p \beta_n}{ p \beta_n}.
\end{align*}
\end{lemma}
\begin{proof} Post the $\tau^{n,C}_k$-th step, i.e., post hitting 
the common state we consider the evolution of the joint process of coalesced 
dual paths and we observe that an exit from the common state can happen only 
through a branching event. Further, branching event can occur only 
if the coalesced dual paths hit a lattice point in $\Z^2$ which is a $\beta$
renewal step. A branching event occurs at a $\beta$ renewal step with probability 
$p \beta_n$. Therefore, post the $\tau^{n,C}_k$-th step, 
the expected number of renewal steps encountered 
before branching is given by $\frac{1 - p\beta_n}{ p \beta_n}$ 
(since branching can occur at the zero-th step as well). Clearly, 
the expected total number of steps required after $\tau^{n,C}_k$
to have a branching is bounded below by $\frac{1 - p\beta_n}{ p \beta_n}$. 
This completes the proof.          
\end{proof}

The next lemma says that, each time the joint process enters to the dependent state, 
it exits from that state very quickly. Further, the probability that the 
next state would be the common state has an uniform lower bound.    
\begin{lemma}
\label{lem:DependentStateTime}
Fix $k \in \N$. We have that 
$$
\E( \tau^{n,D}_{k, \text{next}} - \tau^{n,D}_k  )\leq \frac{1}{(1-p)^3}. 
$$ 
Further, the probability that $\tau^{\epsilon,D}_{k, \text{next}}$ is an entrance time to 
a common state is bounded below by $(1-p)^3$.
\end{lemma}
\begin{proof}
Lemma \ref{lem:CoalescenceProb} gives us that $(\tau^{n,D}_{k, \text{next}}
 - \tau^{n,D}_k)$, i.e., the number of steps elapsed at the $k$-th visit 
 to the dependent state, followed by a visit to the common state,  
 is dominated by a geometric r.v. with success parameter $(1-p)^3$ 
 and this completes the proof.  
\end{proof}

Using, Lemma \ref{lem:CoalescedStateTime} and Lemma \ref{lem:DependentStateTime}, 
we finally show that the total time, i.e., the total number of steps, 
the joint process spends in the dependent state, under diffusive scaling converges to zero in probability.
We introduce some notation. 

Fix $T > 0$. Let $C^n_T, D^n_T, I^n_T : [0,\lfloor T /\epsilon_n^2  \rfloor] \mapsto R^+$ 
respectively defined on the integer points for $j \in [1, T/\epsilon_n^2 ]\cap \Z$ as 
\begin{align}
\label{def:TimeProcesses}
C^n_T(j) & := \sum_{k = 1}^{j} \mathbf{1}_{\{ k-th \text{ step of the joint process is }\text{ taken in common manner}\}} , \nonumber\\
D^n_T(j) & := \sum_{k = 1}^{j} \mathbf{1}_{\{ k-th \text{ step of the joint process is}
 \text{ taken in dependent manner}\}}  ,                \nonumber \\
I^n_T(j) & := \sum_{k = 1}^{j} \mathbf{1}_{\{ k-th \text{ step of the 
joint process taken in independent manner}\}} ,
\end{align} 
and for non-integer points they are joined linearly. For any $j \in \N$, 
$C^{n}_{T}(j)$ denotes the total number of steps of the joint process taken 
in common manner, or the total number of steps that the joint process spends in
the common state during the first $j$ steps. Similarly, $I^{n}_{T}(j)$ and $D^{n}_{T}(j)$
are defined. As a result we have 
$$
I^{n}_{T}(j) +  D^{n}_{T}(j) + C^{n}_{T}(j) = j \text{ for any } j \in [0,T /\epsilon_n^2]\cap \N.
$$
The corresponding scaled versions are denoted respectively  by 
$$
C_{n,T} := S_n(C^n_T), 
D_{n,T} := S_n(D^n_T) \text{ and }
I_{n,T} := S_n(I^n_T).
$$
By definition we have $C_{n,T}, D_{n,T}, I_{n,T} : [0,T] \mapsto \R^+$. 
Because of our discrete setting, for any $ s \in [0,T]$ we have 
$$
|C_{n,T}(s) +  D_{n,T}(s) + I_{n,T}(s) - s| \leq 3\epsilon_n.
$$  
\begin{lemma}
\label{lem:DependentStateTotalTime}
Fix any $T> 0$. As $n \to \infty$, for any $s \in [0,T]$ we have 
$$
D_{n,T}(s) \to 0 \text{ in probability}.
$$
\end{lemma}
\begin{proof}
It is enough to prove Lemma \ref{lem:DependentStateTotalTime} for $s = T$.
This part of the proof is motivated from \cite{EFS17}. Since $C_{n,T}(s)\leq T$ for all $s \leq T$, Lemma \ref{lem:CoalescedStateTime} gives us that 
the expected number of visits to the common state in 
$\lfloor T/\epsilon_n^2 \rfloor$ many steps is bounded above by   
$ T (p \beta_n) /((1 - p \beta_n)\epsilon_n^2)$. By Lemma \ref{lem:DependentStateTime},
 the probability that each visit to the dependent state 
 exits with a visit to the common state is  
bounded below by $(1-p)^3$. As a result, the expected number of 
visits to the common state is at least $(1-p)^3$ 
times the expected number of visits to the dependent state.
Thus the expected number of visits to the dependent state is at most 
$ (T p  \beta_n)/((1 - p \beta_n) (1-p)^3 \epsilon_n^2) $. By Lemma \ref{lem:DependentStateTime} we also have that for each such visit to the dependent state, 
 expected {\it scaled} duration is uniformly bounded above by 
$\epsilon_n^2 / (1-p)^3$. Therefore, (\ref{eq:Limit_b}) gives us that 
$$
\E(D_{n,T}(T)) \leq  ( \frac{T p \beta_n }{(1 - p \beta_n) (1-p)^3 \epsilon_n^2}  \times \frac{\epsilon_n^2}{(1-p)^3 } ) \to 0,
$$
as $n \to \infty$. This completes the proof.
\end{proof}

\section{Proof of Theorem \ref{thm:DoubleNet_Conv}}
\label{sec:ConvDoubleNet}

In this section we will prove Theorem \ref{thm:ConvLeftRight_IntersecTime_Jt} 
to complete the proof of Theorem \ref{thm:DoubleNet_Conv}. In order to prove 
Theorem \ref{thm:ConvLeftRight_IntersecTime_Jt} we need to prove 
some preliminary results.

We consider $\x_1 , \x_2 \in \R^2$ with $\x_1(2) = \x_2(2)$ and w.l.o.g. we assume that 
$\x_1(2) = \x_2(2) = 0$. Corresponding to the choice of $\x_1 , \x_2$, we consider  
$\x^n_{1} , \x^n_{2} \in \widehat{V}$ 
with $\x^n_{1}(2) = \x^n_{2}(2)$  such that 
$$
S_n(\x^n_{i}) \to \x^i\text{ for }i = 1,2\text{ as }n \to \infty.
$$ 
Let $\widehat{\pi}^{i}_r = \widehat{\pi}^{\x^n_i}_r$ denote 
the right dual path starting from $\x^n_i$ for $i = 1,2$. 
Similarly, $\widehat{\pi}^{i}_l$ is defined. 
Let $\{ \widehat{Z}^n(t) : t \geq 0 \}$ denote the process defined as
$$
\widehat{Z}^n(m) =  \bigl (  \sum_{j = 0}^{ m} \mathbf{1}_{\{ \widehat{\pi}^{1}_r(j) = \widehat{\pi}^{2}_l(j) \}} \bigr ) ,
$$
for $m \in \{0\}\cup \N$ and for all non-integer points, we join them linearly. 
Similarly, for the forward process we can consider a sequence of points $\{ \y^n_{1} , \y^n_{2} : n \geq 1\} 
\subseteq \Z^2$ such that the sequence $\{S_n(\y^n_{1}, \y^n_{2}) : n \geq 1\}$ 
converges to $(\y_1, \y_2) \in \R^2$. 
We consider the sequence of rightmost and leftmost forward paths 
$$
\{ (\pi^{1}_r, \pi^{2}_l) = (\pi^{\y^n_1}_r, \pi^{\y^n_2}_l) : n \geq 1 \}
$$ 
and the process $\{ Z^n(t) : t \geq 0 \}$  is similarly defined.
For a pair of leftmost and rightmost paths  (forward as well as dual), 
the next proposition gives joint convergence to left-right Brownian motions 
together with their intersection time.
\begin{proposition}
\label{prop:ConvLeftRight_Intersec}
For any $b \in [0, \infty)$, as $n \to \infty$, the process $S_n \bigl( \widehat{\pi}^{1}_{r} , \widehat{\pi}^{2}_{l} , \widehat{Z}^n \bigr )$ converges in distribution to a limiting process $(\widehat{X}_1, \widehat{X}_2, \widehat{Z})$, where $(\widehat{X}_1, \widehat{X}_2)$ 
is distributed as left-right dual Brownian motions with drift $\bar{b}$, where $\bar{b}$ 
is as in (\ref{def:Drift_Param}),  starting from $\x_1$ and $\x_2$ respectively and
we have
$$
\widehat{Z}(t) = \int_0^t \mathbf{1}_{\{ \widehat{X}_1(s) = \widehat{X}_2(s) \}}ds .
$$  
A similar statement holds for the sequence $\{ S_n(\pi^{1}_{r} , \pi^{2}_{l},
Z^n) : n \geq 1\}$ as well. 
\end{proposition}
Before proving this proposition, we would like to highlight the following remark
which says that Proposition \ref{prop:ConvLeftRight_Intersec} gives joint 
convergence of {\it first} intersection times as well. 

\begin{remark}
\label{rem:ConvLeftRight_Intersec}
For $b > 0$ it was observed in  \cite{SSY24} Section 3.1 
that since $(\widehat{X}_1-\widehat{X}_2)$ is a sticky Brownian motion with
stickiness at $0$,the convergence of $S_n(\widehat{Z}^n)$ to $\widehat{Z}$ 
in distribution implies that any intersection time of $\widehat{X}_1$ and $\widehat{X}_2$ lies in the support of $d\widehat{Z}$. Therefore, Proposition \ref{prop:ConvLeftRight_Intersec} gives us the joint convergence of $S_n(\widehat{\pi}^1, \widehat{\pi}^2)$ and their (scaled) first intersection time to $\widehat{X}_1$ and $\widehat{X}_2$ and their first intersection time. In addition for $b = 0$ 
we observe that Proposition \ref{prop:ConvLeftRight_Intersec} 
gives joint convergence to forward as well as to dual coalescing Brownian paths 
(without any drift) and their first intersection times or coalescing times.
\end{remark}



\subsection{A coupling construction}
To prove Proposition \ref{prop:ConvLeftRight_Intersec} we need to first describe a coupling construction. We consider three independent collections of i.i.d. random vectors given as 
\begin{align*}
& \psi^{(a)} := \{\Lambda^{(a)}_{(x,t)} := ( B^{(a)}_{(x,t)}, R^{(a)}_{(x,t)}, 
\text{Br}^{n,(a)}_{(x,t)} ) : (x,t) \in \Z^2\}, \\
& \psi^{(b)} := \{\Lambda^{(b)}_{(x,t)} := (B^{(b)}_{(x,t)}, R^{(b)}_{(x,t)}, 
\text{Br}^{n,(b)}_{(x,t)}) : (x,t) \in \Z^2 \} \text{ and }\\
& \psi^{(c)} := \{\Lambda^{(c)}_{(x,t)} := (B^{(c)}_{(x,t)}, R^{(c)}_{(x,t)}, 
\text{Br}^{n,(c)}_{(x,t)}) : (x,t) \in \Z^2\}.
\end{align*}
These collections will be used to a produce a coupling construction such that the
rightmost and the leftmost dual paths will be coupled with their {\it independent} counterparts. 
Corresponding to these collections we define $V^{(a)} := \{ (x,t) \in \Z^2 : B^{(a)}_{(x,t)} = 1\}$. Similarly, the collections $V^{(b)}$ and $V^{(c)}$ are defined.  
Based on the collection $V^{(a)}$ and for any $(x,t) \in \Z^2$, the r.v.s 
$J^{+,(a)}_{(x,t)}$ and $J^{-,(a)}_{(x,t)}$ are defined as 
$$
J^{+,(a)}_{(x,t)} := \inf\{k \geq 1 : (x+k,t) \in V^{(a)}\} \text{ and }
J^{-,(a)}_{(x,t)} := \inf\{k \geq 1 : (x-k,t) \in V^{(a)}\}.
$$ 
The random variables 
$J^{+,(b)}_{(x,t)}, J^{-,(b)}_{(x,t)}, J^{+,(c)}_{(x,t)}, J^{-,(c)}_{(x,t)}$ are defined similarly. 
With a slight abuse of notation for $(x,t) \in \Z/2\times \Z$ 
with $x \notin \Z$, we define $J^{+,(a)}_{(x,t)}$ and $J^{-,(a)}_{(x,t)}$  as
\begin{align*}
J^{+,(a)}_{(x,t)} & := \inf\{ k \geq 1 : (x - 1/2 + k, t) \in V^{(a)} \} \text{ and }\\
J^{-,(a)}_{(x,t)} & := \inf\{ k \geq 1 : (x + 1/2 - k, t) \in V^{(a)} \}.
\end{align*} 
The random variables $J^{+,(b)}_{(x,t)}, J^{-,(b)}_{(x,t)}, J^{+,(c)}_{(x,t)}$
and $J^{-,(c)}_{(x,t)}$ are defined similarly.

Fix $(y,0), (x,0) \in \R^2$ with $x < y$. 
We consider $(y^n, 0), (x^n, 0) \in \widehat{V}$ such that 
$$
S_n ( y^n , x^n ) \to (y, x) \text{ as }n \to \infty.
$$
In order to start the dual steps, $\widehat{h}_{n, (l)}(y^n,0)$ 
and $\widehat{h}_{n, (r)}(x^n,0)$, we will use the common collection 
$\psi^{\text{mixed}}_1 := \{ \Lambda^{\text{mixed}}_{(u,-1)} : u \in \Z \}$
which we are going to construct now. Note that $y^n \neq x^n$ with $x^n < y^n$. 
On the event $\{ x^n + J^{+, (a)}_{(x^n , -1)} < y^n - J^{-, (b)}_{(y^n , -1)}\}$ 
we define the collection $\psi^{\text{mixed}}_1$ as below:
\begin{align*}
\Lambda^{\text{mixed}}_{(u,-1)} := 
\begin{cases}
\Lambda^{(a)}_{(u,-1)} & \text{ if } u \leq x^n + J^{+,(a)}_{(x^n , -1)} \\
\Lambda^{(b)}_{(u,-1)} & \text{ if } u \geq y^n - J^{-,(b)}_{(y^n , -1)} \\
\Lambda^{(c)}_{(u,-1)} & \text{ otherwise}.
\end{cases}
\end{align*}
On the complementary event, i.e., on the event $\{ x^n + J^{+, (a)}_{(x^n , -1)} 
\geq  y^n - J^{-,(b)}_{(y^n , -1)}\}$ we observe that the event 
 $\{ J^{+,(a)}_{(x^n , -1)} > j^1_1 , J^{-,(b)}_{(y^n , -1)} > j^1_2 \}$
 must occur for some $j^1_1, j^1_2 \geq 0$ with 
 $y^n - x^n = j^1_1 + j^1_2 + 1$. On the said event we define the collection 
 $\psi^{\text{mixed}}_1$ as below:
\begin{align*}
\Lambda^{\text{mixed}}_{(u,-1)} := 
\begin{cases}
\Lambda^{(a)}_{(u,-1)} & \text{ if } u \leq x^n + j^1_1 \\
\Lambda^{(b)}_{(u,-1)} & \text{ if } u \geq y^n - j^1_2 \\
\Lambda^{(c)}_{(u,-1)} & \text{ otherwise}.
\end{cases}
\end{align*}
Based on the collection $\psi^{\text{mixed}}_1$, the respective 
dual steps are denoted as $\widehat{h}_{(r)}(x^n,0), 
\widehat{h}_{(l)}(y^n,0)$. Next, we are going to define the common collection 
$\psi^{\text{mixed}}_2 := \{ \Lambda^{\text{mixed}}_{(u, -2)} : u \in \Z \}$ 
which will be used to obtain the next step of the joint process. 

On the event $\{ \widehat{h}_{(r)}(x^n,0) = \widehat{h}_{(l)}(y^n,0)\}$ 
the common collection is simply taken as 
$\psi^{\text{mixed}}_2 := \{ \Lambda^{\text{mixed}}_{(u, -2)} = 
\Lambda^{(c)}_{(u, -2)} : u \in \Z \}$.
Next we consider the situation $\widehat{h}_{(r)}(x^n,0) < \widehat{h}_{(l)}(y^n,0)$.
Assuming $(\widehat{h}_{(r)}(x^n,0), \widehat{h}_{(l)}(y^n,0)) = (x^n_1, y^n_1)$,
on the event $\{ x^n_1 + J^{+,(a)}_{(x^n_1,-2)} < y^n_1 - J^{-,(b)}_{(y^n_1,-2)} \}$ 
we define the collection $\psi^{\text{mixed}}_2$ as below:
\begin{align*}
\Lambda^{\text{mixed}}_{(u,-2)} := 
\begin{cases}
\Lambda^{(a)}_{(u,-2)} & \text{ if } u \leq x^n_1 + J^{+,(a)}_{(x^n_1 , -2)} \\
\Lambda^{(b)}_{(u,-2)} & \text{ if } u \geq y^n_1 - J^{-,(b)}_{(y^n_1 , -2)} \\
\Lambda^{(c)}_{(u,-2)} & \text{ otherwise}.
\end{cases}
\end{align*}
As earlier, on the complementary event, i.e., on the event 
$\{ x^n_1 + J^{+,(a)}_{(x^n_1,-2)} \geq y^n_1 - J^{-,(b)}_{(y^n_1,-2)} \}$ 
we observe that 
 $\{ J^{+,(a)}_{(x^n_1 , -2)} > j^2_1 , J^{-,(b)}_{(y^n_1 , -2)} > j^2_2 \}$
 must occur for some $j^2_1, j^2_2 \geq 0$ with 
 $y^n_1 - x^n_1 = j^2_1 + j^2_2 + 1$. On the said event we define the collection 
 $\psi^{\text{mixed}}_2$ as below:
\begin{align*}
\Lambda^{\text{mixed}}_{(u,-2)} := 
\begin{cases}
\Lambda^{(a)}_{(u,-2)} & \text{ if } u \leq x^n_1 + j^2_1 \\
\Lambda^{(b)}_{(u,-2)} & \text{ if } u \geq y^n_1 - j^2_2 \\
\Lambda^{(c)}_{(u,-2)} & \text{ otherwise}.
\end{cases}
\end{align*} 
The corresponding step for the joint process, which is the second step, 
is denoted by $\widehat{h}^2_{(r)}(x^n,0), \widehat{h}^2_{(r)}(y^n,0)$.   
For general steps we can iterate the same coupling construction by method of induction. 

Given the $k$-th step $(\widehat{h}^k_{(r)}(x^n,0), 
\widehat{h}^k_{(r)}(y^n,0)) = (x^n_k , y^n_k)$ we now define
the collection $\psi^{\text{mixed}}_{k+1} := 
\{ \Lambda^{\text{mixed}}_{(u, -k-1)} = \Lambda^{(c)}_{(u, -k-1)} : u \in \Z \}$
which will be used to obtain the next dual step. 
On the event $\{ \widehat{h}^k_{(r)}(x^n,0) = \widehat{h}^k_{(l)}(y^n,0)\}$ 
the common collection is simply taken as 
$\{ \Lambda^{\text{mixed}}_{(u, -k-1)} = 
\Lambda^{(c)}_{(u, -k-1)} : u \in \Z \}$.
Next we consider the situation $\widehat{h}^k_{n, (r)}(x^n,0) < 
\widehat{h}^k_{n, (l)}(y^n,0)$.
Assuming $(\widehat{h}^k_{n, (r)}(x^n,0), \widehat{h}^k_{n, (r)}(y^n,0)) 
= (x^n_k, y^n_k)$, on the event $\{ x^n_k + J^{+,(a)}_{(x^n_k,-k-1)} < 
y^n_k - J^{-,(b)}_{(y^n_k,-k-1)} \}$ we define the collection 
$\psi^{\text{mixed}}_{k+1}$ as below:
\begin{align*}
\Lambda^{\text{mixed}}_{(u,-k-1)} := 
\begin{cases}
\Lambda^{(a)}_{(u,-k-1)} & \text{ if } u \leq x^n_k + J^{+,(a)}_{(x^n_k , -k-1)} \\
\Lambda^{(b)}_{(u,-k-1)} & \text{ if } u \geq y^n_k - J^{-,(b)}_{(y^n_k , -k-1)} \\
\Lambda^{(c)}_{(u,-k-1)} & \text{ otherwise}.
\end{cases}
\end{align*}
As earlier, on the complement event, i.e., on the event 
$\{ x^n_k + J^{+,(a)}_{(x^n_k,-k-1)} \geq y^n_k - J^{-,(b)}_{(y^n_k,-k-1)} \}$
we observe that the event 
 $\{ J^{+,(a)}_{(x^n_k , -k-1)} > j^{k+1}_1 , J^{-,(b)}_{(y^n_k , -k-1)} > j^{k+1}_2 \}$
 must occur for some $j^{k+1}_1, j^{k+1}_2 \geq 0$ with 
 $y^n_1 - x^n_1 = j^{k+1}_1 + j^{k+1}_2 + 1$. On the said event we define the collection 
 $\psi^{\text{mixed}}_{k+1}$ as below:
\begin{align*}
\Lambda^{\text{mixed}}_{(u,-k-1)} := 
\begin{cases}
\Lambda^{(a)}_{(u,-k-1)} & \text{ if } u \leq x^n_k + j^{k+1}_1 \\
\Lambda^{(b)}_{(u,-k-1)} & \text{ if } u \geq y^n_k - j^{k+1}_2 \\
\Lambda^{(c)}_{(u,-k-1)} & \text{ otherwise}.
\end{cases}
\end{align*} 
The corresponding step for the joint process, which is the second step of 
the joint process and denoted by $\widehat{h}^{k+1}_{(r)}(x^n,0), 
\widehat{h}^{k+1}_{(l)}(y^n,0)$.
 
Following Section \ref{sec:Div_TotalTime}, we partition the total evolution time of the joint process obtained through the coupling construction in three categories as follows: 
\begin{definition} 
We say that the $j+1$-th step of the (coupled) joint process is   
\begin{itemize}
\item[(i)] is taken in common manner or in  the common state 
if $\widehat{h}^{j}_{(l)}(y^n,0) = 
\widehat{h}^{j}_{(r)}(x^n,0)$;
\item[(ii)] in independent manner or in the independent state if $\widehat{h}^{j}_{(r)}(x^n,0) \neq 
\widehat{h}^{j}_{(l)}(y^n,0)$ and the event 
$$
\bigl \{ \widehat{h}^{j}_{(r)}(x^n,0)(1) + J^{+,(a)}_{(\widehat{h}^{j}_{(r)}(x^n,0)(1),-j-1)}
< \widehat{h}^{j}_{(l)}(y^n,0)(1) - J^{-,(a)}_{(\widehat{h}^{j}_{(l)}(y^n,0)(1),-j-1)} 
\bigr \} 
$$ 
occurs;
\item[(iii)] in dependent manner or in the dependent state if $\widehat{h}^{j}_{(r)}(x^n,0) \neq 
\widehat{h}^{j}_{ (l)}(y^n,0)$ and we have the occurrence of the event 
$$
\bigl \{ \widehat{h}^{j}_{(r)}(x^n,0)(1) + J^{+(a)}_{(\widehat{h}^{j}_{(r)}(x^n,0)(1),-j-1)}
\geq \widehat{h}^{j}_{(l)}(y^n,0)(1) - J^{-(a)}_{(\widehat{h}^{j}_{(l)}(y^n,0)(1),-j-1)} 
\bigr \} .
$$
\end{itemize} 
\end{definition}
Set $\tau^{n, C}_{\text{coupled},0} = \tau^{n,C}_{\text{coupled},0} = 
\tau^{n,C}_{\text{coupled},0} = 0$. For $j \geq 1$, 
the $j$-th re-entrance time to the common state for the coupled joint process 
has been defined as 
\begin{align*}
\tau^{n,C}_{\text{coupled},j} & :=  
\inf \{ k + 1 > \tau^{n,C}_{\text{coupled},j - 1 } :  
(k+1) \text{-th step is in common state and }k\text{-th step is not} \}.
\end{align*}
Similarly,  for $j \geq 1$, the steps $\tau^{n,D}_{\text{coupled},j}$ and 
$\tau^{n,I}_{\text{coupled},j}$ are defined as the $j$-th entrance times to respective states. Our coupling construction readily implies the next corollary.
\begin{corollary}  
\label{cor:TimeDivisionDist_Equal}
We have 
$$
\bigl \{ \bigl ( \tau^{n,D}_{\text{coupled},j}, \tau^{n,I}_{\text{coupled},j}
, \tau^{n,C}_{\text{coupled},j} \bigr ) : j \geq 1\bigr \} \stackrel{d}{=} 
\bigl \{ \bigl ( \tau^{n,D}_{j}, \tau^{n,I}_{j}
, \tau^{n,C}_{j} \bigr ) : j \geq 1 \bigr \} .
$$
\end{corollary}

Let $\{ \beta^{(l)}_\ell = \beta^{(l)}_\ell(y^n,0) : \ell \geq 1 \}$
denote the sequence of renewal steps for the process 
$\{ \widehat{h}^{j}_{(l)}(y^n,0) : j \geq 1\}$. Similarly, the sequence  
$\{ \beta^{(r)}_\ell = \beta^{(r)}_\ell(x^n,0) : \ell \geq 1 \}$ denotes 
the sequence of renewal steps for the process 
$\{ \widehat{h}^{j}_{(r)}(x^n,0) : j \geq 1\}$.

We mention here that our proof idea is motivated from that of \cite{EFS17}.
We would like to consider the evolution of a dual path while being in the independent 
state during different visits and represent it as an evolution of 
the $\alpha$-type dual path Markov chain over the total time spent by the joint process in 
the independent state. We would like to  apply similar techniques for 
multiple visits to the common state as well. However, there is one crucial issue. 
As the $\alpha$-type dual path is a Markov chain
and at each step the increment depends on whether the chain starts from a site in 
$\Z$ or not, we cannot simply combine the increments of the underlying 
Markov chain over multiple visits to the independent (coalesced) state 
and represent it as evolution of the same Markov chain 
over the total time spent in that state. To deal with this, during each visit
to the coalesced state or to the independent state, we consider 
the time spent being within renewal blocks 
{\it only}. The remaining steps in these two states are adjusted as being spent 
in the {\it extended} dependent state. Fix any $T > 0$ and to 
keep track of the proposed modification we need 
to update the definitions of the processes $I^n_T, D^n_T$ and $C^n_T$ 
as follows.

For $n \in \N$ the processes  
$\bar{I}^{n,(r)}_T, \bar{I}^{n,(l)}_T, \bar{D}^{n,(r)}_T, \bar{D}^{n,(l)}_T$ and 
$\bar{C}^{n}_T$ are defined over the domain $[0, \lfloor T/\epsilon^2 \rfloor ]$ as below.
For any $t \in [0, \lfloor T/\epsilon^2 \rfloor]\cap \N$  
\begin{align*}
\bar{I}^{n,(r)}_T(t) & := \sum_{k = 1}^t \mathbf{1}_{\{ k \in [\beta^{(r)}_\ell , \beta^{(r)}_{\ell + 1}] \text{ and } [\beta^{(r)}_\ell , \beta^{(r)}_{\ell + 1}]
\subset [\tau^{n,I}_{k}, \tau^{n,I}_{k,\text{next}})
\text{ for some }\ell,k \geq 1\}} \\
\bar{I}^{n,(l)}_T(t) & := \sum_{k = 1}^t \mathbf{1}_{\{ k \in [\beta^{(l)}_\ell , 
\beta^{(l)}_{\ell + 1}] \text{ and } [\beta^{(l)}_\ell , \beta^{(l)}_{\ell + 1}]
\subset [\tau^{n,I}_{k}, \tau^{n,I}_{k,\text{next}})
\text{ for some }\ell,k \geq 1\}} \\
\bar{C}^{n}_T(t) & := \sum_{k = 1}^t \mathbf{1}_{\{ k \in [\beta^{(r)}_{\ell_1} , 
\beta^{(r)}_{\ell_1 + 1}] \text{ and } [\beta^{(r)}_{\ell_1} , \beta^{(r)}_{\ell_1 + 1}]
\subset [\tau^{n,C}_{k}, \tau^{n,C}_{k,\text{next}})
\text{ for some }\ell_1,k \geq 1 \}}\\
& \qquad \qquad = \sum_{k = 1}^t \mathbf{1}_{\{ k \in [\beta^{(l)}_{\ell_2} , 
\beta^{(l)}_{\ell_2 + 1}] \text{ and } [\beta^{(l)}_{\ell_2} , 
\beta^{(l)}_{\ell_2 + 1}] \subset [\tau^{n,C}_{k}, \tau^{n,C}_{k,\text{next}})
\text{ for some }\ell_2, k \geq 1 \}} \\
\bar{D}^{n,(r)}_T(t) & := \sum_{k = 1}^t \bigl ( 1 - \mathbf{1}_{\{ \bar{I}^{n,r}_T(t)
- \bar{I}^{n,r}_T(t-1) = 1\}}\vee \mathbf{1}_{\{ \bar{C}^{n}_T(t) - 
\bar{C}^{n}_T(t-1) = 1\}} \bigr )\text{ and }\\
\bar{D}^{n,(l)}_T(t) & := \sum_{k = 1}^t \bigl ( 1 - \mathbf{1}_{\{ 
\bar{I}^{n,l}_T(t) > \bar{I}^{n,l}_T(t-1) \}} \vee \mathbf{1}_{\{ 
\bar{C}^{n}_T(t) > \bar{C}^{n}_T(t-1)\}} \bigr ) .
\end{align*}
By definition, for any $\alpha \in \{ l , r\}$ and for any $t \in [0,T]\cap \N$
we have $\bar{I}^{n,(\alpha)}_T(t) + \bar{C}^{n}_T(t) + \bar{D}^{n,(\alpha)}_T(t) = t$.
Based on the collection $\psi^{(a)}$ and for $\alpha \in \{l, r\}$, we consider 
the Markov chain $\{ M^{(a)}_{(\alpha),k} : k \geq 1 \}$ which represents
the sequence of $x$ coordinates of successive steps of 
the $\alpha$-type dual path starting from $(0,0)$. Similarly, the Markov chains  
$\{ M^{(b)}_{(\alpha),k} : k \geq 1 \}$ and $\{ M^{(c)}_{(\alpha),k} : k \geq 1 \}$ 
are defined.  
We observe that by construction, for any $\alpha \in \{l, r\}$ the three Markov chains, $\{ M^{(a)}_{(\alpha),k} : k \geq 0\}$
and $\{ M^{(b)}_{(\alpha),k} : k \geq 0\}$ and $\{ M^{(c)}_{(\alpha),k} : k \geq 0\}$ 
are independent. We are now ready to proceed with the proof of Proposition \ref{prop:ConvLeftRight_Intersec}.

\noindent {\bf Proof of Proposition \ref{prop:ConvLeftRight_Intersec}:}
Here we prove Proposition \ref{prop:ConvLeftRight_Intersec} for dual paths. 
The argument for leftmost and rightmost forward paths is similar and in fact, it is slightly 
simpler. In Remark \ref{rem:ForwardPathConvergence} 
we mention the minor modification that is required to deal with forward paths.   
Considering the evolution of the process $\{ \widehat{h}^{k}_{(r)}(x^n,0) : k \geq 0 \}$ in the coupled construction, 
the functions $\bar{I}^{n,r}_T, \bar{D}^{n,r}_T, \bar{C}^{n}_T$
partitions the evolution during $\lfloor T/\epsilon^2_n \rfloor$ many steps into three categories of which, we say that $\bar{D}^{n,r}_T$ denotes the total time spent
in the {\it extended} dependent state during first $\lfloor T/\epsilon^2_n \rfloor$ many steps. We need to analyse the evolution of the joint process during 
its tenure in the {\it extended} dependent state. The total time spent
in the {\it extended} dependent state by the process 
$\{ \widehat{h}^{k}_{(l)}(y^n,0) : k \geq 0 \}$ is given by $\bar{D}^{n,l}_T$.
We highlight that the time spent in the extended 
dependent state by the process $\{ \widehat{h}^{k}_{(l)}(y^n,0) : k \geq 0 \}$
need not be same as that of for the process $\{ \widehat{h}^{k}_{(r)}(x^n,0) : k \geq 0 \}$.
The evolution of the joint process during the stay at 
extended dependent state is defined as the processes $\{ 
\mathcal{Y}^{\text{dep}}_{(l)}(j), \mathcal{Y}^{\text{dep}}_{(r)}(j) : j \geq 1 \}$
which are defined as 
\begin{align*}
\mathcal{Y}^{\text{dep}}_{(l)}(j) & := \sum_{k=0}^{j-1} \mathbf{1}_{\{\bar{D}^{n,(l)}_T(k+1)>
\bar{D}^{n,(l)}_T(k) \}} (\widehat{h}^{k+1}_{(l)}(y^n,0)(1) - 
\widehat{h}^{k}_{(l)}(y^n,0)(1)) \text{ and } \\
\mathcal{Y}^{\text{dep}}_{(r)}(j) & := \sum_{k=0}^{j-1} \mathbf{1}_{\{\bar{D}^{n,(r)}_T(k+1)>
\bar{D}^{n,(r)}_T(k) \}} (\widehat{h}^{k+1}_{(r)}(x^n,0)(1) - \widehat{h}^{k}_{(r)}(x^n,0)(1)). 
\end{align*}

We define the scaled version of the 
processes $\bar{C}^{n}_{T},  \bar{D}^{n,(\alpha)}_T , \bar{I}^{n,(\alpha)}_T$
as
$$
(\bar{C}_{n,T}, \bar{D}^{(\alpha)}_{n,T}, \bar{I}^{(\alpha)}_{n,T}) := 
S_n ( \bar{C}^{n}_{T},  \bar{D}^{n,(\alpha)}_T , \bar{I}^{n,(\alpha)}_T)
\text{ for }\alpha \in \{l, r\}.
$$
Our coupling construction gives us that for the joint distribution of the scaled 
 dual paths  at any $s \in [0,T]$ we have:  
\begin{align}
\label{q:SDE_ScaledPaths_1}
\widehat{L}_n(s) & = S_n 
\Bigl( y^n + M^{(b)}_{(l)}( \bar{I}^{n,(l)}_T(\lfloor s/\epsilon_n^2 \rfloor))   +
M^{(c)}_{(l)}(\bar{C}^n_T(\lfloor s/\epsilon_n^2 \rfloor))   
+ \mathcal{Y}^{\text{dep}}_{(l)}(\bar{D}^{n,(l)}_T( \lfloor s/\epsilon_n^2 \rfloor )) \Bigr ) \nonumber \\
\widehat{R}_n(s) & = S_n 
\Bigl(  x^n + M^{(a)}_{(r)}( \bar{I}^{n,(r)}_T(\lfloor s/\epsilon_n^2 \rfloor))   +
M^{(c)}_{(r)}(\bar{C}^n_T(\lfloor s/\epsilon_n^2 \rfloor))   
+ \mathcal{Y}^{\text{dep}}_{(r)}(\bar{D}^{n,(r)}_T( \lfloor s/\epsilon_n^2 \rfloor )) \Bigr ) \nonumber \\
0 & = \int_{0}^s \mathbf{1}_{ \{ \widehat{R}_n(u) > \widehat{L}_n(u)\}} 
d \bar{C}_{n,T}(u)\text{ and }
\end{align}
for any $s \in [0,T]$ we have 
\begin{align}
\label{eq:SDE_ScaledPaths_2}
|s  - (\bar{C}_{n,T}(s) +  \bar{D}^{n,(\alpha)}_T(s) + 
\bar{I}^{n,(\alpha)}_T(s))| \leq 3\epsilon_n\text{ for }\alpha \in \{ l, r \}.
\end{align}
We reiterate that unlike the coalescing system of branching paths in \cite{SS08} or in \cite{EFS17}, in this work we are dealing with paths with long range interactions.

We first focus on arguing that the limit $\lim_{n \to \infty} S_n((\mathcal{Y}^{\text{dep}}_{(l),j}, \mathcal{Y}^{\text{dep}}_{(r),j}))$ exists and understanding it. 
We acknowledge that we don't have a precise understanding of the joint distribution 
of this process. Fortunately, for our purpose is not required and we will follow the approach 
outlined in \cite{EFS17} as discussed below. 
We recall $\bar{b} = b\frac{(1-p)}{(2-p)^2}$ (see (\ref{def:Drift_Param})). 
Since, increments of each of these unscaled processes are bounded by $1$, 
it is not difficult to see that the processes $\{ \bar{C}_{n,T}, \bar{D}^{(r)}_{n,T} ,
\bar{D}^{(l)}_{n,T}, \bar{I}^{(r)}_{n,T}, \bar{I}^{(l)}_{n,T} : n \geq 1\}$
are jointly tight. The next lemma is a stronger version of Lemma \ref{lem:DependentStateTotalTime} and gives us that under diffusive scaling, the total time 
spent in the {\it extended} dependent state converges to $0$.   
\begin{lemma}
\label{lem:ExtDependentStateTotalTime}
We have 
\begin{align*}
\lim_{n \to \infty} \bar{D}^{(r)}_{n,T} = \lim_{n \to \infty} \bar{D}^{(l)}_{n,T} \to 0\text{ in probability},
\end{align*}
where $0$ denotes the constant zero map defined on $[0,T]$.
\end{lemma}
\begin{proof}
W.l.o.g. we show that $\bar{D}^{(r)}_{n,T} \to 0$ in probability 
as $n \to \infty$. The same argument holds for  $\bar{D}^{(l)}_{n,T}$ too.
Lemma \ref{lem:DependentStateTime}, gives us that 
\begin{align*}
D_{n,T} \to 0 \text{ as }n \to \infty \text{ in probability}.
\end{align*}
Recall that $\bar{D}^{(r)}_{n,T}$ is the total time spent in the extended dependent state
by the right dual path. Note that after each entrance 
to the independent state, exit must happen necessarily 
through an entrance to the dependent state only. Therefore, the number of entrances to 
the independent state  can be at most one more than the number of entrances to the dependent state and as a result,  the scaled number of expected entrances to 
the independent state in $\lfloor T/\epsilon^2 \rfloor $ many steps converges to zero.
Similarly, each entrance to the common state must precede 
through an exit from the dependent state only. 
Therefore,  the scaled number of expected entrances to 
the common state in $\lfloor T/\epsilon^2 \rfloor $ many steps converges to zero.
Further, during each visit either to the independent state or to the common state, 
there can be some residual steps, which would be adjusted 
and counted towards the $\bar{D}^{n,(r)}_{T}$ function. 
We show that under diffusive scaling, the contribution of these scaled residual 
steps is zero. We present the argument for a visit to the independent 
state and similar argument holds for a visit to the commons state. 
    
These residual steps are precisely of two categories: 
\begin{itemize}
\item[(i)] Intermediate steps between the entrance to the independent state 
and hitting the next renewal step for the process 
$\{ \widehat{h}^{k}_{(r)}(x^n,0) : k \geq 0\}$ while being in the same 
(independent) state and  
\item[(ii)] intermediate steps between hitting the last renewal step 
for the process $\{\widehat{h}^{k}_{(r)}(x^n,0) : k \geq 0\}$ within that visit and exit from the independent state (i.e., entrance to the dependent state).
\end{itemize}
By Proposition \ref{prop:DualRenewalSteps}, we have that the expected numbers of 
both types of residual steps (both can be zero also) are bounded by the expected renewal time which is bounded by uniform constants depending only on process parameters. 
Therefore, during each visit scaled number of
expected residual steps (of both types) converges to zero. Therefore,
under diffusive scaling, the expected total time spent in the extended dependent 
state converges to zero and this completes the proof. 
\end{proof}

Because of Lemma \ref{lem:ExtDependentStateTotalTime}, it is enough to 
show that the sequence $\{ S_{n}(\mathcal{Y}^{\text{dep}}_{(l)}(j), 
\mathcal{Y}^{\text{dep}}_{(r)}(j)) : n \geq 1 \}$ is tight as that would 
directly imply that 
$$
S_{n} \bigl( \{ \mathcal{Y}^{\text{dep}}_{(l)}(j), 
\mathcal{Y}^{\text{dep}}_{(r)}(j) : j \geq 1\} \bigr ) \rightarrow (0,0)
\text{ as }n \to \infty,
$$
where, $0$ denotes the constant zero map  over the {\it domain $\{ 0\}$}. 
Tightness of the sequence $\bigl \{ S_{n} \bigl ( \{\mathcal{Y}^{\text{dep}}_{(l)}(j), 
\mathcal{Y}^{\text{dep}}_{(r)}(j) : j \geq 1\}                                                                                                                                                                                                                                            \bigr ) : n \geq 1 \bigr \}$ follows from tightness 
of the family $\{ S_{n}(\NN^n, \widehat{\NN}^n): n \geq 1 \}$
which holds due to Theorem \ref{thm:ConvLeftRight_Jt} (see \cite{SVZ24}).



We mentioned that the processes, $\bar{C}^n_T, \bar{I}^{n,(r)}_T$ and $ \bar{I}^{n,(l)}_T$, are  obtained through linear interpolation and 
their increments are bounded by one. Hence, as mentioned earlier, 
the sequence $\{(\bar{C}_{n,T}, \bar{I}^{(r)}_{n,T}, \bar{I}^{(l)}_{n,T}, 
\bar{D}^{(r)}_{n,T}, \bar{D}^{(l)}_{n,T}) : n \geq 1\}$ is tight and  (\ref{eq:SDE_ScaledPaths_2}) ensures that for any sub-sequential limit  
$(C_T, I^{(r)}_T, I^{(l)}_T, 0, 0)$ we have  
$$
I^{(r)}_T(s) + C_T(s) = s = I^{(l)}_T(s) + C_T(s)\text{ for all }s \in [0,T],
$$ 
implying $I^{(r)}_T = I^{(l)}_T$.
Finally, the earlier discussion together with Proposition \ref{prop:DualPathDiffLimit} suggest that the sequence 
\begin{align}
\label{eq:TightSubseq_1}
\Bigl \{ \Bigl ( \widehat{L}_n, \widehat{R}_n, C_{n,T}, I^{(r)}_{n,T}, I^{(l)}_{n,T}, D^{(r)}_{n,T},
D^{(l)}_{n,T} , S_n(M^{(b)}_{(l)},M^{(a)}_{r}, M^{(c)}_{(l)}, M^{(c)}_{r})\Bigr ) : 
n \geq 1 \Bigr \}
\end{align}
is tight and, as $n \to \infty$, a general sub sequential limit of (\ref{eq:TightSubseq_1})
is given by  
$$
\bigl ( \widehat{L}, \widehat{R}, C_T, I_T, I_T, 0, 0, \widehat{B}_{(a)}^l, 
\widehat{B}^r_{(b)}, \widehat{B}_{(c)} \bigr ),
$$ 
where $\widehat{B}^l$, $\widehat{B}^r$ are independent backward Brownian motions with 
drift $\bar{b}$ and $-\bar{b}$ respectively 
and $\widehat{B}_c$ denotes a standard backward Brownian motion independent of 
$\widehat{\bar{B}}^l$ and $\widehat{\bar{B}}^r$.
In order to prove the first part of Proposition \ref{prop:ConvLeftRight_Intersec}, 
it suffices to show that the subsequential limit vector 
$( \widehat{L} , \widehat{R} , C_T , I_T, \widehat{B}^l, \widehat{B}^r, \widehat{B}^c)$ 
solves the system of SDE's  in (\ref{eq:LR_SDE_new}) (modified for system of backward paths), i.e., we need to show that
\begin{align*}
 d \widehat{L}(s) & =  d \widehat{B}^l_{I_T(s)} +  d \widehat{B}^c_{C_T(s)} + \bar{b} ds ,\\
 d \widehat{R}_s & =  d \widehat{B}^r_{I_T(s)} + d \widehat{B}^c_{C_T(s)} - \bar{b} ds ,\\
 s & = I_T(s) + C_T(s) , \\
 0 & = \int_0^s \mathbf{1}_{\{ \widehat{L}(s) >  \widehat{R}(s)\}} dC_T(s).
\end{align*}
We have proved that $I_T(s) + C_T(s) = s$ for all $s \in [0,T]$. 
In order to complete the proof we need to show 
that the limiting $C$ process indeed represents the {\it intersection time} process
for $(\widehat{L}(s), \widehat{R}(s))_{s \in [0,T]}$. By construction, we must have 
$$
\widehat{L}_n(s) \geq \widehat{R}_n(s) \text{ for all }s \in [0,T], 
$$  
and therefore, for the limiting processes $(\widehat{L}(s), \widehat{R}(s))_{s \geq 0}$ we have
$\widehat{L}(s) \geq \widehat{R}(s)$  for all $s \in [0,T]$. Therefore,
in order to complete the proof it suffices to show that we have  
$$
\int_0^s \mathbf{1}_{\{ \widehat{L}(s) > \widehat{R}(s)\}} dC_T(s) = 0.
$$ 
In order to do that we need to work a little more. 
We consider the integral $\mathcal{I}^{(n)}$ defined as 
\begin{align}
\label{eq:Integral}
\mathcal{I}^{(n)} = \int_0^T \mathbf{1}_{\{(\widehat{L}_n(s) - \widehat{R}_n(s))> \epsilon_n \}} d\bar{C}_{n,T} (s). 
\end{align}
We observe that a necessary condition for the r.h.s. in (\ref{eq:Integral}) to 
contribute is that the process $\bar{C}_{n,T}$ must change over continuous time too. 
In this regard, we examine the behaviour of unscaled  function $\bar{C}^n_T$ 
over the interval $[k, k+1]$ where $(k+1)$-th step of the corresponding unscaled joint process is one of the following:
\begin{itemize}
\item[(a)] $(k+1)$-th step is an entrance step to the common state, i.e., 
 $k$-th step is in the dependent state and $(k+1)$-th step is in the common state.
 We observe that in this situation by definition of these states, 
 the two paths were distinct at the $(k-1)$-th step and they have coalesced 
 at the $k$-th step. 
\item[(b)] $(k+1)$-th step is an exit step from the common state, i.e., 
 $(k+1)$-th step is in the common state (taken in common manner) and $(k+2)$-th step is 
 in a different state. Then by definition, the two paths were remain coalesced 
 at the $k$-th step and branched away at the $(k+1)$-th step.
\end{itemize} 
Firstly, we observe that if the $(k+1)$-th step of the joint process 
is not among the above mentioned two types, then over the interval $(k , k+1)$
either the process $\bar{C}^n_T$ remains {\it constant} or the Indicator function 
considered in the integrand of (\ref{eq:Integral}) equals zero. Correspondingly, the integrand in the r.h.s. of (\ref{eq:Integral}) does not contribute for 
$(k \epsilon_n^2, (k+1)\epsilon_n^2)$. 

For Item (a) by definition, we must have that the two paths have intersected (and coalesced) at the $k$-th step. We assume that the $(k+1)$-th step is not an exit step 
from the common state, i.e., $(k+2)$-th step is in the common state too which means 
the two paths remain coalesced at the $(k+1)$-th step as well. 
Otherwise, $(k+1)$-th step would be an exit step from the common state as well and 
$(k+2)$-th step would be in some {\it other} state. We will consider 
this case in Item (b). If two paths remain coalesced till the $(k+1)$-th step, 
then the indicator function on the r.h.s. of (\ref{eq:Integral}) 
vanishes over the interval $(k \epsilon_n^2, (k+1)\epsilon_n^2)$.
On the other hand, for Item (b) the process $\bar{C}^n_T$ increases 
linearly by one over the interval $[k, k+1]$ as the $(k+1)$-th step 
is in the common state. Since the $(k+1)$-th step is an exit from
the common state, the $(k+2)$-th step must be in some other state and that would mean 
we must have a branching at the $(k+1)$-th step. Consequently, because of branching 
the two paths are at least one distance apart from each other at the $(k+1)$-th step. Therefore, we may have non-zero contribution in the r.h.s. of (\ref{eq:Integral}) 
and the contribution is dominated by 
$\mathbf{1}_{(k \epsilon_n^2, (k+1)\epsilon_n^2)}
\times \mathbf{1}_{\{\text{branching occurs at the } (k+1) \text{-th for the coalesced path}\}}$. 

Therefore, the integrand considered in 
(\ref{eq:Integral}) makes non-zero contributions 
only at the branching steps, i.e., steps where coalesced paths branch out. 
We mention here that this behaviour is slightly different from the models considered in 
\cite{SS08} or in \cite{EFS17} where possible contributions happen at the 
entrance steps to the coalesced (which is different from our common) state. 
The reason for this unusual behaviour is that the entrance and 
exit from the common state respectively run in one lag with the entrance and 
exit from the coalesced state. This was highlighted in Remark \ref{rem:ForwardPathConvergence}.

Clearly, the number of exits from the common state is bounded by the number 
of entrances into the common state. We 
mentioned in Remark \ref{rem:DepState_CoalState} that 
every entrance to the common state must precede by an 
exit from the dependent state and as observed earlier, the number of entrances to the 
dependent state is clearly dominated by the total time spent in the dependent state. Therefore, we have 
   $$ 
   D_{n,T}(T) \geq \int_0^T \mathbf{1}_{\{(\widehat{L}_n(s) - \widehat{R}_n(s))> \epsilon_n \}}  
   dC_{n,T} (s) .
   $$

Finally let $\rho_\delta(x) =0 $ if $0\leq x \leq \delta$ and $ \rho(x) = 1$ if $x \geq \delta$.
For $\epsilon_n < \delta$, we obtain 
\begin{align}
\label{eq:InteractionRightLeftPath}
D_{n,T}(T) & \geq \int_0^T 1_{\{ \widehat{L}_n(s) - \widehat{R}_n(s))> \epsilon_n \} }  
dC_{n,T} (s) \nonumber \\
& \geq  \int_0^T \rho_\delta ( \widehat{L}_n(s)- \widehat{R}_n(s)) dC_{n,T}(s).
\end{align}
Taking the limit as $n \to \infty$ in (\ref{eq:InteractionRightLeftPath}) we get 
$0  \geq  \int_0^T \rho_\delta  (\widehat{L}(s)-\widehat{R}(s))dC(s)$
which implies
$$
\int_0^T  \rho_\delta  (\widehat{L}(s)-\widehat{R}(s))dC(s) = 0 .
$$
Since this is true for all $\delta>0$ we are done.
\qed 

Next we show that for $\alpha \in \{l, r\}$ a pair of $\alpha$-type 
dual paths converge to respective backward coalescing Brownian motions with appropriate 
drifts and their first intersection time, which is same as the coalescing time, jointly 
converge too. We need to introduce some notations.   

Fix $\y_1, \y_2 \in \R^2$ with $\y_1(2) = \y_2(2) = 0$. 
Choose $\{(\y^n_1, \y^n_2) : n \geq 1 \} \subset \widehat{V}$ such that
$\y^n_1(2) = \y^n_2(2) = 0$ and $S_n(\y^n_1, \y^n_2) \to (\y_1, \y_2)$
as $n \to \infty$. For $\alpha \in 
\{ l, r \}$, let $\widehat{\pi}^i_{(\alpha)}$ denote the $\alpha$ type dual path 
starting from $\y^n_i$. Let $\widehat{\tau}_{(\alpha)}$ denote the first intersection time, 
which is the coalescing time as well, between $\widehat{\pi}^1_{(\alpha)}$ and
$\widehat{\pi}^2_{(\alpha)}$. Unlike the first intersection time between left-right dual paths, it can be proved that the r.v. $\widehat{\tau}_{(\alpha)}$ is finite a.s. 
\begin{proposition}
\label{prop:ConvLeftLeft_Coalesc}
For any $b \in [0, \infty)$, as $n \to \infty$, the process $S_n ( 
\widehat{\pi}^{1}_{\alpha} , \widehat{\pi}^{2}_{\alpha}, \widehat{\tau}_{(\alpha)} \bigr )$ converges in distribution to a limiting process $(\widehat{X}^1_{\alpha}, \widehat{X}^2_{\alpha}, \widehat{\tau}_{(\alpha)})$, where $(\widehat{X}^1_{\alpha}, \widehat{X}^2_{\alpha})$ denotes $\alpha$-type 
coalescing dual Brownian motions with drift $\bar{b}$  starting from $\y_1, \y_2$ 
respectively and $\widehat{\tau}_{(\alpha)}$ denotes the coalescing time. 

A similar statement holds for the sequence of scaled $\alpha$-type forward paths as well. 
\end{proposition}

\begin{proof}
The argument is exactly same as that of Proposition \ref{prop:ConvLeftRight_Intersec} and 
therefore, we only provide brief heuristics here. Fix any $T > 0$ and it is enough 
to show that the scaled $\alpha$-type dual path trajectories restricted over the time 
interval $[0,T]$ together with their scaled coalescing time converge to 
the claimed limit. As in the proof of Proposition 
\ref{prop:ConvLeftRight_Intersec}, Lemma \ref{lem:DependentStateTotalTime} gives  
us that the for the scaled process the total amount of time being spent 
in the dependent state in the time interval $[0,T]$ converges to $0$ and therefore, 
the proof follows. Note that for the said process, the coalesced state acts as an 
absorbing state.    
\end{proof}

Before we proceed further, we recall that we prove both the Propositions
essentially for dual paths. Below in Remark \ref{rem:Tightness} we mention 
the minor modification required to deal with forward paths and justify that 
both propositions hold for forward paths as well. 
\begin{remark}  
\label{rem:ForwardPathConvergence}
For any $(x,t) \in \Z^2$ it is not difficult to observe that 
the process $\{ h^j_{(\alpha)}(x,t)(1) : j \geq 1 \}$ is a random walk with i.i.d. 
increments with mean $b_{(\alpha)}$ and variance $\sigma = \sigma(p)$ (defined as in \ref{def:Diff_Coeff}), where 
\begin{align*}
b_{(\alpha)} = 
\begin{cases}
+ \beta_n\frac{(1-p)}{(2-p)^2} & \text{ if }\alpha = r \\
- \beta_n\frac{(1-p)}{(2-p)^2} & \text{ if }\alpha = l .
\end{cases}
\end{align*} 
For a proof we refer the reader tp Lemma 5.3 of \cite{SVZ24}.
Further, given that  $h^j_{(\alpha)}(x,t)  = (v, t + j ) \in V$, 
for the $(j+1)$-th step, the exploration region $E^{j+1}_{(x,t)}$ is given by 
$$
E^{j+1}_{(x,t)} := \{ (w, t + j + 1) \in \Z^2 : 
v - J^-_{(v,t+1)} \le w \le v + J^+_{(v,t+1)}\}
$$
as the step $h_{(\alpha)}(v,t + j) = h^{j+1}_{(\alpha)}(x,t)$ is measurable 
w.r.t. the $\sigma$-field 
$$
\sigma \bigl ( B_{(w,t+j+1)}, 
R_{(w,t+j+1)}, \text{Br}_{(v, t + j)} : w \in \Z, 
v - J^-_{(v,t+1)} \le w \le v + J^+_{(v,t+1)} \bigr ).
$$ 
For $(x_1,t), (x_2,t) \in \Z^2$ considering the joint process
$\{(h^j_{(\alpha_1)}(x_1,t), h^j_{(\alpha_2)}(x_1,t)) : j \geq 0 \}$
we observe that as long as the respective exploration regions don't overlap, 
both the processes evolve independently. This observation allows us to
partition the total time of evolution of the joint process of two Howard paths
 into three categories 
as in the case of dual paths. Further, straightforward modification the argument of 
of Lemma \ref{lem:CoalescenceProb} gives us that every time the (joint) process 
hits the dependent state, it has a positive probability 
(uniformly bounded away from zero) to jump to the coalesced state at the very next step.           
This allows us to extend Lemma \ref{lem:DependentStateTotalTime} for 
forward paths as well and obtain both the propositions for forward paths.  
\end{remark}

For any $\alpha_1, \alpha_2 \in \{ l, r\}$, Proposition \ref{prop:ConvLeftRight_Intersec} and 
Proposition \ref{prop:ConvLeftLeft_Coalesc} together give joint 
convergence for a pair of $\alpha_1, \alpha_2$ (forward as well as dual) paths
together with their first intersection time.
In order to complete the proof of Theorem \ref{thm:ConvLeftRight_IntersecTime_Jt}
we need to show that the said convergence holds for finitely many $\alpha_1, 
\alpha_2$ type paths as well. As discussed before, this completes the proof 
of Theorem \ref{thm:DoubleNet_Conv}.  

\noindent{\bf Proof of Theorem \ref{thm:ConvLeftRight_IntersecTime_Jt}:}
We only need to show that Proposition \ref{prop:ConvLeftRight_Intersec} 
(as well as Proposition \ref{prop:ConvLeftLeft_Coalesc})
can be extended for finitely many paths. The argument is identical to that of the 
proof of Proposition 5.2 of \cite{SS08}. Therefore, we only present the brief heuristics 
here. We recall that the construction of finitely many left-right Brownian paths, as discussed in Subsection \ref{subsec:BrownianNet} is an inductive construction which is made up of several independent pieces ({\it over time}) consisting of segments
of left paths, right-most paths and pair(s) of left-right paths.
We can follow the same inductive construction for the {\it corresponding} pieces.
The stopping times used in this construction are continuous functionals with
respect to the law of independent evolutions of paths within each such piece. 
As a result, Theorem \ref{thm:ConvLeftRight_IntersecTime_Jt} follows from 
Proposition \ref{prop:ConvLeftRight_Intersec} and Proposition \ref{prop:ConvLeftLeft_Coalesc}. This completes the proof.
\qed

\begin{remark}
\label{rem:Tightness}
Finally, we would like to highlight that Remark \ref{rem:ConvLeftRight_Intersec} highlights that our technique works for the forward Howard paths as well and gives an 
independent efficient proof in terms of converge to the 
left-right Brownian web w.r.t. a {\it stronger} topology (see Theorem \ref{thm:ConvLeftRight_IntersecTime_Jt}). The stated convergence, 
ensures that the sequence $\{ S_n (\NN^n, \widehat{\NN}^n) : n \geq 1\}$ is tight
as well and gives a self-contained proof.  
\end{remark}

\end{document}